\documentclass[11pt,letterpaper]{article}
\usepackage{amsfonts, amsmath, amssymb, amscd, amsthm, color, graphicx, mathrsfs, wasysym, setspace, mdwlist}
\usepackage{hyperref}
\newcommand{\G}{\Gamma (G, X\cup\mathcal H)}
\newcommand{\Hl}{\{ H_\lambda\}_{\lambda \in \Lambda}}
\newcommand{\Lab}{{\bf Lab}}

\newcommand{\e}{\varepsilon}

\newtheorem{thm}{Theorem}[section]
\newtheorem{cor}[thm]{Corollary}
\newtheorem{lem}[thm]{Lemma}
\newtheorem{prop}[thm]{Proposition}
\newtheorem{q}[thm]{Question}

\theoremstyle{definition}
\newtheorem{defn}[thm]{Definition}
\theoremstyle{remark}
\newtheorem{rem}[thm]{Remark}
\newtheorem{ex}[thm]{Example}

\newcommand{\la}{\langle}
\newcommand{\ra}{\rangle}

\newfont{\eufm}{eufm10}

\begin{document}

\title{$C^\ast$-simple groups without free subgroups}

\author{A.Yu. Olshanskii, D.V. Osin\thanks{This work was partially supported by the RFBR grant 11-01-00945.
The first author was also supported by
the NSF grants DMS-1161294. The second author was also supported by
the the NSF grant DMS-1308961.}}

\date{}

\maketitle

\begin{abstract}
We construct first examples of non-trivial groups without non-cyclic free subgroups whose reduced $C^\ast$-algebra is simple and has unique trace. This answers a question of de la Harpe. Both torsion and torsion free examples are provided. In particular, we show that the reduced $C^\ast$-algebra of the free Burnside group $B(m,n)$ of rank $m\ge 2$ and any sufficiently large odd exponent $n$ is simple and has unique trace.
\end{abstract}

\renewcommand{\thefootnote}{\fnsymbol{footnote}}
\renewcommand{\thefootnote}{\arabic{footnote}}

\tableofcontents


\section{Introduction}


Over the past few decades, there has been considerable interest in simple $C^\ast$-algebras in both group theoretic and  operator algebraic communities. For a comprehensive survey of the recent research in this area we refer to \cite{Har07} and \cite{Ror}. The main goal of our paper is to provide new examples of groups whose reduced $C^\ast$-algebras are simple and have unique trace. We begin by recalling some basic definitions and relevant results.

The \emph{reduced $C^\ast$-algebra} of a group $G$, denoted $C^\ast_{red}(G)$, is the closure of the linear span of $\{\lambda_G(g)\mid g\in G\}$ with respect to the operator norm, where $\lambda_G\colon G\to U(\ell^2(G))$ denotes the left regular representation. A (non-zero) $C^\ast $-algebra is said to be \emph{simple} if it contains no proper non-trivial two-sided closed ideals. A group $G$ is called \emph{$C^\ast$-simple} if $C^\ast_{red}(G)$ is simple.

$C^\ast$-simplicity of a group is essentially a representation theoretic property. Indeed it can be characterized in terms of weak containment of unitary representations introduced by Fell \cite{Har07}. For non-trivial groups, $C^\ast$-simplicity can be also thought of as a strong negation of amenability. Indeed it is not hard to show that the amenable radical of $G$ is trivial whenever $G$ is $C^\ast$-simple; in particular, if $G$ is both $C^\ast$-simple and amenable, then $G=\{ 1\}$.

Closely related to $C^\ast$-simplicity is the uniqueness of trace on $C^\ast_{red}(G)$.  Recall that a (normalized) \emph{trace} on a unitary C$^\ast $-algebra $A$ is a positive linear functional $\tau\colon A\to \mathbb C$ such that $\tau (1)=1$ and $\tau(ab)=\tau (ba)$ for all $a, b \in A$. The reduced $C^\ast$-algebra of every group $G$ has a canonical trace (see Section 2.1). If it is the only trace, $G$ has several nice dynamical and group theoretic properties, e.g., shift minimality  and the absence of non-trivial amenable invariant random subgroups \cite{TD}. In turn, any one of these two properties implies triviality of the amenable radical. For the sake of completeness, we should mention that the exact relation between uniqueness of trace on $C^\ast_{red}(G)$, $C^\ast$-simplicity, and triviality of the amenable radical of $G$ is rather mysterious; in particular, it is still unknown whether all of these properties are equivalent or not. For more details we refer to \cite{TD}

The class of groups having simple reduced $C^\ast$-algebras with unique trace includes many examples acting ``nontrivially" on a hyperbolic space: non-virtually cyclic  hyperbolic and relatively hyperbolic groups without non-trivial finite normal subgroups \cite{AM,Har88}, centerless mapping class groups of closed surfaces, $Out(F_n)$ for $n\ge 2$ \cite{BriHar}, and many $3$-manifold groups and fundamental groups of graphs of groups \cite{HP}. Most of these results can be generalized in the context of acylindrically hyperbolic groups (see \cite{DGO,Osi13} for details). Examples of completely different nature are provided by $PSL_n(\mathbb Z)$ and, more generally, lattices in connected semi-simple centerless Lie groups without compact factors \cite{BCH}. Finally we mention a recent result from \cite{PT} stating that torsion free groups satisfying a weak form of the Atiyah Conjecture  and having positive first $\ell^2$-Betti number are $C^\ast$-simple.

Simplicity of $C^\ast_{red}(G)$ and uniqueness of trace are usually derived from certain algebraic or dynamical properties of the group $G$. Typical examples include the Powers property (and its weak versions) \cite{Har07,TD} and  property (PH) of Promislow \cite{Pro}; Brin and Picioroaga noticed that these properties also imply the existence of non-cyclic free subgroups. The proof of the latter fact has appeared in \cite{Har07} (see the remark following Question 15) and in \cite[Theorem 5.4]{TD}. More algebraic approaches, such as the one suggested by Akemann and Lee \cite{AL} or the property $P_{nai}$ introduced by Bekka, Cowling, and de la Harpe in \cite{BCH}, imply the existence of non-cyclic free subgroups immediately.

This motivates the following questions asked by de la Harpe \cite[Question 15]{Har07}:

\begin{q}\label{dlHQ}
\begin{enumerate}
\item[(a)] Does there exist a (non-trivial) $C^\ast$-simple group without non-cyclic free subgroups?
\item[(b)] Is the free Burnside group of rank at least $2$ and sufficiently large odd exponent $C^\ast$-simple?
\end{enumerate}
\end{q}

One can think of these questions as variations of the classical Day--von Neumann problem, which asks whether every non-amenable group contains a non-cyclic free subgroup. The negative answer was obtained by the first author in \cite{Ols80}. Later Adyan \cite{A} proved non-amenability of free Burnside groups of sufficiently large odd exponent and since then many other counterexamples have been found.

Although the original Day--von Neumann question has negative answer, one can still hope to obtain a positive result by strengthening the non-amenability assumption or by relaxing the ``no non-cyclic free subgroups" condition. In the last 15 years, many results -- both negative and affirmative -- were obtained in this direction \cite{GR,LO,MonOza,OS,Osi11,Osi09,W}. Since $C^\ast$-simplicity can be regarded as a strong negation of amenability, Question \ref{dlHQ} fits well in this context.

\paragraph{Main results.} The main goal of our paper is to give the affirmative answer to both parts of the de la Harpe's question. Let $B(m,n)$ denote the free Burnside group of rank $m$ and exponent $n$. That is, $B(m,n)$ is the free group of rank $m$ in the variety of groups satisfying the identity $X^n=1$. In Section 3, we prove the following.

\begin{thm}\label{main1}
For every $m\ge 2$ and every sufficiently large odd $n$, the reduced $C^\ast$-algebra of $B(m,n)$ is simple and has unique trace.
\end{thm}

The basic idea behind the proof of Theorem \ref{main1} is that simplicity and uniqueness of trace of $C^\ast_{red}(B(m,n))$ can be derived from certain algebraic properties of $B(m,n)$. This part of our paper essentially uses the technique developed by the first author in \cite{book} as well as some new ideas (e.g., towers and repelling sections in van Kampen diagrams, see Section \ref{trs}). As a by-product, we obtain some groups theoretic facts about free Burnside groups which seem to be of independent interest (e.g., Corollary \ref{Burns-cor}).

In Section 4, we suggest another way of constructing $C^\ast$-simple groups without non-cyclic free subgroups. It makes use of the methods from \cite{Osi10} and \cite{Osi07}, namely small cancellation theory and Dehn filling in relatively hyperbolic groups.  This approach is independent of the results about free Burnside groups and allows us to construct both torsion and torsion free examples. It is noteworthy that, unlike in the case of the original Day--von Neumann problem, there is no obvious way of getting torsion free $C^\ast$-simple groups without non-cyclic free subgroups from torsion ones. Indeed if one has a non-amenable group $G$ without non-cyclic free subgroups and $G=F/N$, where $F$ is a free group, then it is easy to show that the group $F/[N,N]$ is torsion free, has no non-cyclic free subgroups, and is non-amenable. However $F/[N,N]$ is never $C^\ast$-simple as its amenable radical is non-trivial.

\begin{thm}\label{main3}
For any non-virtually cyclic  hyperbolic group $H$ and any countable group $C$ without non-cyclic free subgroups, there exists a quotient group $G$ of $H$ such that
\begin{enumerate}
\item[(a)] $G$ has no non-cyclic free subgroups;
\item[(b)] $C$ embeds in $G$;
\item[(c)] $C_{red}^\ast (G)$ is simple and has unique trace.
\end{enumerate}
Moreover, if $C$ is torsion (respectively, $C$ and $H$ are torsion free), then $G$ can be made torsion (respectively, torsion free) as well.
\end{thm}

This theorem can be used to construct groups without non-cyclic free subgroups that combine $C^\ast$-simplicity with other strong negations of amenability. For example, taking $H$ to be a hyperbolic group with Kazhdan property (T) (e.g., a lattice in $Sp(n,1)$) and $C$ to be a non-unitarizable group without non-cyclic free subgroups, we obtain a group $G$ which is also non-unitarizable and has property (T) in addition to all properties listed in Theorem \ref{main3}. For examples of non-unitarizable groups without non-cyclic free subgroups and the discussion of the related Dixmier problem we refer to \cite{Osi09} or \cite{MonOza}.

In another direction, using uncountability of the sets of all finitely generated torsion and torsion free groups without non-cyclic free subgroups we obtain the following.

\begin{cor}\label{main2}
\begin{enumerate}
\item[(a)] There exist $2^{\aleph_0}$ non-isomorphic finitely generated torsion groups $G$ such that $C^\ast_{red}(G)$ is simple with unique trace.
\item[(b)] There exist $2^{\aleph_0}$ non-isomorphic finitely generated torsion free groups $G$ without non-cyclic free subgroups such that $C^\ast_{red}(G)$ is simple with unique trace.
\end{enumerate}
\end{cor}

All $C^\ast$-simple groups constructed in this paper are inductive (co)limits of sequences of $C^\ast$-simple hyperbolic (or relatively hyperbolic) groups and epimorphisms $H_1\to H_2\to \ldots$. It is worth noting that simplicity of such limits is not automatic. Indeed, there exist such sequences of $C^\ast$-simple hyperbolic groups whose limits are even amenable (see Example \ref{non-simple}).

\paragraph{Structure of the paper and advice to the reader.}
We begin by recalling basic analytic definitions in Section 2.1. In Section 2.2 we introduce the notion of a sequence of groups with infinitesimal spectral radius and exhibit some examples.  The main results are Corollary \ref{bisr} and Proposition \ref{Pij}, which are necessary for the proof of Theorems \ref{main1} and \ref{main3}, respectively. The proofs of all new results in this section are quite elementary.

Theorems \ref{main1} and \ref{main3} are proved in Sections 3 and 4, respectively, by using a sufficient condition for $C^\ast$-simplicity and uniqueness of trace formulated by Akemann and Lee (see Lemma \ref{AL}). To apply this condition to a group $G$ we need to make sure that norms of certain sequences of elements in $C^\ast_{red}(G)$ converge to $0$. The latter condition can be deduced from purely algebraic properties, namely Theorem 3.21 and Proposition 4.15. Most of Sections 3 and 4 is devoted to the proof of these two results.

The proof of Theorem 3.21 is given in Section 3 and essentially relies on the geometric technique developed by the first author in \cite{book}. We provide a brief introduction and include (a simplified versions of) basic definitions and main technical lemmas in Section 3.1. We believe that this introduction is sufficient to understand  Sections 3.2 and 3.3. For a more detailed account we refer to Chapters 5 and 6 of \cite{book}. The reader who is willing to believe in technical group theoretic results and only wants to understand the main idea of the proof of Theorem \ref{main1} can go directly to the statement of Theorem 3.21 and read the rest of Section 3.3; all arguments there are self-contained and elementary modulo Theorem 3.21 and results of Section 2.

The proof of Proposition 4.15 is given in Section 4 and is based on papers \cite{Osi10} and \cite{Osi07}. We include a brief review of necessary definitions and results from these papers. The final argument that derives Theorem \ref{main3} from Proposition 4.15 and results of Section 2 is essentially the same as the corresponding argument in the proof of Theorem \ref{main1}.

Finally we note that Sections 3 and 4 are completely independent and can be read separately.

\paragraph{Acknowledgments.} We would like to thank P. he la Harpe and the referee for corrections and useful comments.

\section{$C^\ast$-simplicity and sequences of groups with infinitesimal spectral radius}


\subsection{Analytic preliminaries}
Given a (discrete) group $G$, we denote by $\ell^2(G)$ the set of all square-summable functions $f\colon G\to \mathbb C$ and by $\lambda_G\colon G\to \mathcal U(\ell^2(G))$ its left regular representation. The left regular representation extends by linearity to the representation of the group algebra $\mathbb CG\to B(\ell^2(G))$, where $B(\ell^2(G))$ denotes the algebra of all bounded operators on $\ell^2(G)$; we keep the same notation $\lambda_G$  for this extension.

For a function $f\colon G\to \mathbb C$,  $\| f\|_p$  denotes its $\ell^p$-norm. We also denote by $\| A\| $ the operator norm of $A \in  B(\ell^2(G))$. It is well-known and easy to prove that
\begin{equation}\label{2n1}
\| a\| _2\le \|\lambda_G(a)\| \le \|a\|_1
\end{equation}
for every $a\in \mathbb CG$.

Recall that the \emph{reduced $C^\ast $-algebra of a group $G$}, denoted $C^\ast_{red}(G)$, is the closure of $\lambda_G (\mathbb CG)$ in $ B(\ell^2(G))$ with respect to the operator norm. The involution on $C^\ast_{red}(G)$ is induced by the standard involution on $\mathbb CG$:
$$\left(\sum_{g\in G} \alpha_g g\right)^\ast = \sum_{g\in G} \bar\alpha_g g^{-1},$$ where $\alpha _g\in \mathbb C$ for all $g\in G$.

A group $G$ is called \emph{$C^\ast$-simple} if $C^\ast_{red}(G)$ has no non-trivial proper two-sided ideals. Since $C^\ast_{red}(G)$ is unital, this is equivalent to the absence of non-trivial proper two-sided closed ideals.

A (normalized) \emph{trace} on a unital $C^\ast $-algebra $\mathcal A$ is a linear functional $\theta\colon \mathcal A \to \mathbb C$ which satisfies $\theta (1)=1$, $\theta (A^\ast A)\ge 0$, and $\theta (AB)=\theta(BA)$ for all $A,B\in \mathcal A$. For every group $G$, $C^\ast_{red} (G)$ has a canonical trace $\tau\colon C^\ast_{red} \to \mathbb C$, which extends the map $\mathbb CG\to \mathbb C$ given by
\begin{equation}\label{tauCG}
\tau \left(\sum\limits_{g\in G} \alpha_gg\right) = \alpha_1.
\end{equation}
For a general element $A\in C^\ast_{red}(G)$, the trace can be defined by using the inner product in $\ell^2(G)$ as follows:
$$
\tau (A)=\langle A\delta_1, \delta_1\rangle,
$$
where $\delta _1\in \ell^2(G)$ is the characteristic function of $\{ 1\}$.

It is well-known that for every $a\in \mathbb CG$, one has
\begin{equation}\label{Kes}
\left\| \lambda_G (a)\right\| = \limsup\limits_{n\to\infty} \sqrt[2n]{\tau(\lambda_G(a^\ast a)^n)}
\end{equation}
(see \cite{Kes}). The following results are immediate consequences of (\ref{tauCG}) and (\ref{Kes}). By $\mathbb R_+$ we denote the set of all non-negative real numbers.

\begin{lem}\label{norms}
Let $G$ be a group.
\begin{enumerate}
\item[(a)] Let $H\le G$, $a\in \mathbb CH$. Then $\left\| \lambda_H(a)\right\|= \left\| \lambda_G(a)\right\|$.
\item[(b)] Let $\e\colon G\to Q$ be a homomorphism. We keep the same notation for the natural extension $\mathbb CG\to \mathbb CQ$. Then for every $a\in \mathbb R_+G$, we have $\left\| \lambda_G(a)\right\| \le \left\| \lambda _Q (\e(a))\right\|$.
\item[(c)] Let $a=\sum_{g\in G} \alpha_g g$, and $b= \sum_{g\in G} \beta_g g$ be elements of $\mathbb R_+G$  such that $\alpha_g\le \beta_g$ for every $g\in G$. Then we have $\left\| \lambda_G(a)\right\| \le \left\| \lambda_G(b)\right\|$.
\end{enumerate}
\end{lem}

Given a unital $C^\ast$-algebra $\mathcal A$, let $U(\mathcal A)=\{ u\in \mathcal A\mid uu^*=u^*u=1_A\} $ denote the group of unitary elements.
Recall that a unital $C^\ast$-algebra $\mathcal A$ has the \emph{Dixmier property} if
\begin{equation}\label{DC}
\overline{conv}\{ uau^*\mid u\in U(\mathcal A) \} \cap \mathbb C 1_A \ne \emptyset
\end{equation}
for every $a\in \mathcal A$, where $\overline{conv}$ denotes the closure of the convex hull. It is not hard to show that if a unital $C^\ast $-algebra $\mathcal A$ satisfies the Dixmier property and has a trace, then it is simple and has unique trace \cite{Dix}. Conversely, if a unital $C^\ast $-algebra is simple with unique trace, then it satisfies the Dixmier property \cite{HZ}.

To prove uniqueness of trace and simplicity of the reduced $C^\ast$-algebra of a group $G$ it suffices to verify (\ref{DC}) for $a=\lambda_G (g)$ for every $g\in G$. More precisely, one has the following lemma.

\begin{lem}\label{AL}
Let $G$ be a countable group. Suppose that there exists a sequence of finite subsets $\mathcal Y=\{ Y_i\}_{i\in \mathbb N}$ of $G$ such that
\begin{equation}\label{eq:AL}
\lim\limits_{i\to \infty}\frac1{|Y_i|} \left\|\sum\limits_{y\in Y_i}\lambda_G(ygy^{-1})\right\| =0
\end{equation}
for every $g\in G\setminus \{ 1\}$. Then $C^\ast _{red}(G)$ is simple with unique trace.
\end{lem}

The proof is quite elementary and the idea goes back to the Powers' paper \cite{Pow}. To the best of our knowledge, this lemma was first explicitly stated by Akemann and Lee in \cite{AL}. In one form or another, it is a part of most existent proofs of $C^\ast$-simplicity and our paper is not an exception.

In general, computing the operator norm of an element of $C^\ast_{red}(G)$ is rather difficult. The standard way to overcome this problem is the following. First note that by part (a) of Lemma \ref{norms}, to compute the operator norm in (\ref{eq:AL}) we can only look at the subgroup generated by the set $\{ ygy^{-1}\mid y\in Y_i\}$ and ignore the rest of $G$. Further, most standard approaches utilize certain algebraic or dynamical properties of $G$ to show that the subsets $Y_i$ can be chosen so that, loosely speaking, the set $\{ ygy^{-1}\mid y\in Y_i\}$ is a basis of a free subgroup of $G$. Then (\ref{eq:AL}) follows from the well-known formula $$\frac1i\left\| \lambda_{F_i}(x_1)+\cdots +\lambda_{F_i}(x_i)\right\|=\frac{2\sqrt{i-1}}i \to 0$$ as $i\to \infty$, where ${F_i}$ is the free group with basis $x_1, \ldots, x_i$, and $i\ge 2$ \cite{AO}.

Of course, this approach cannot be used to construct a $C^\ast$-simple group without free subgroups. However a similar idea can be implemented if instead of free groups $F_i$ and their bases one uses sequences of (non-free) groups $G_i$ and their generating sets $X_i$ such that $$\frac1{|X_i|}\left\| \sum \limits_{x\in X_i} \lambda_{G_i}(x)\right\| \to 0$$ as $i\to \infty$. This leads to the notion of a sequence of groups with infinitesimal spectral radius discussed below.

\subsection{Sequences of groups with infinitesimal spectral radius}
Given a group $G$ and a finite subset $X=\{ x_1, \ldots, x_n\} \subseteq G$, let
\begin{equation}\label{AS}
A_X=\frac1n \big(\lambda_G (x_1)+\cdots+\lambda_G(x_n)\big).
\end{equation}
If $X$ generates $G$, the spectral radius of $$A_{X^{\pm 1}}:=\frac12(A_X+ A_{X^{-1}}),$$ where $X^{-1}=\{ x^{-1}\mid x\in X\}$, is usually denoted by $\rho(G,X)$ and called the \emph{spectral radius of the simple random walk} on $G$ with respect to $X$. Since the operator $A_{X^{\pm 1}}$ is self-adjoint, we have $\rho(G,X)=\left\| A_{X^{\pm 1}}\right\| $.

Further let $G=F(X)/N$, where $F(X)$ is a free group with basis $X$ and $N\lhd F(X)$. The associated \emph{cogrowth function} of $G$ is defined by
$$
\gamma (n)= |\{ g\in N\mid {\rm dist}_X(1,g)\le n\}|,
$$
where ${\rm dist}_X$ is the word distance with respect to $X$. Let $\omega (G,X)$ denote the \emph{cogrowth rate} of $G$ with respect to $X$, that is,
$$
\omega (G,X)= \lim\limits_{n\to \infty}\sup \sqrt[n]{\gamma (n)}.
$$
If $N\ne \{ 1\}$, the cogrowth rate is related to the spectral radius by the Grigorchuk formula, see \cite{Grig}.
\begin{equation}\label{Grig}
\rho (G,X)=\frac{\sqrt{2|X|-1}}{2|X|} \left( \frac{\sqrt{2|X|-1}}{\omega(G,X)} + \frac{\omega(G,X)}{\sqrt{2|X|-1}} \right)
\end{equation}

Finally we recall that the \emph{Cheeger constant} of a group $G$ with respect to a generating set $X$, denoted $h(G,X)$, is defined by the formula
$$
h(G,X)= \inf_F \frac{|\partial F|}{|F|},
$$
where the infimum is taken over all finite subsets $F\subseteq G$ and $\partial F$ denotes the set of all edges in the unoriented Cayley graph $\Gamma(G,X)$ of $G$ with respect to $X$ that connect vertices from $F$ to vertices from $G\setminus F$. Here by the \emph{unoriented Cayley graph} $\Gamma =\Gamma(G,X)$ of $G$ with respect to a generating set $X$ we mean the graph with vertex set $V(\Gamma)=G$ and edge set $E(\Gamma)=\{ (g, gx,x)\mid g\in G,\, x\in X\}$, where the edge $(a,b, x)$ connects vertices $a$ and $b$. Note that $\Gamma $ is always $2|X|$-regular even if $X$ is symmetric or has involutions. (We have to accept this definition of the Cayley graph to make it consistent with the definition of the spectral radius given above.)

The Cheeger constant is related to the spectral radius by the following discrete version of the Cheeger-Buser inequality, see \cite[Theorem 2.1 (a) and Theorem 3.1 (a)]{Moh}.
\begin{equation}\label{Moh}
\frac{2|X|(1-\rho(G,X))}{2|X|-1}\le \frac{h(G,X)}{2|X|}\le \sqrt{1-\rho(G,X)^2};
\end{equation}
note that the spectral radius $\rho $ in \cite{Moh} corresponds to $2|X|\rho(G,X)$ in our notation.

\begin{lem}\label{isr-char}
For any sequence $\mathcal G=\{(G_i,X_i)\}_{i\in \mathbb N}$ of groups $G_i$ and finite generating sets $X_i$, the following conditions are equivalent.
\begin{enumerate}
\item[(a)] $\lim\limits_{i\to \infty}\left\| A_{X_i}\right\| = 0$.
\item[(b)] $\lim\limits_{i\to \infty}\rho(G_i,X_i) = 0$.
\item[(c)] $\lim\limits_{i\to \infty} \frac{h(G_i,X_i)}{2|X_i|}=1$.
\item[(d)] $\lim\limits_{i\to \infty} \frac{\omega (G_i, X_i)}{|X_i|}=0$ and $\lim\limits_{i\to \infty} |X_i|=\infty$.
\end{enumerate}
\end{lem}

\begin{proof}
Applying part (c) of Lemma \ref{norms} to $a=\frac1{|X_i|}\sum_{x\in X_i} x$ and $b=\frac1{|X_i|}\sum_{x\in X_i} (x+x^{-1})$ we  obtain
\begin{equation}\label{A1}
\left\|A_{X_i}\right\| =\left\|\lambda _{G_i}(a)\right\|  \le \left\|\lambda_{G_i}(b)\right\| = 2\left\| A_{X_i^{\pm 1}}\right\|= 2\rho (G_i, X_i).
\end{equation}
On the other hand, we have
\begin{equation}\label{A2}
\rho (G_i, X_i) = \frac12\left\|A_{X_i} + A_{X_i^{-1}} \right\| \le \frac12\left(\left\|A_{X_i} \right\| + \left\|A_{X_i}^\ast \right\| \right) = \left\|A_{X_i}\right\|.
\end{equation}
Obviously (\ref{A1}) and (\ref{A2}) imply that (a) and (b) are equivalent.

Note that by (\ref{2n1}) we have $\rho (G_i, X_i)\ge 1/(2|X_i|)$. Hence (b) implies that $|X_i|\to \infty$ as $i\to \infty$. Now the equivalence of (b) and (c) follows from (\ref{Moh}).

Finally let us prove the equivalence of (b) and (d). Let $I=\{ i\mid {G_i \;\rm is\; free\; with\; basis\; }X_i\}$. The Grigorchuk formula (\ref{Grig}) applies to $(G_i,X_i)$ whenever $i\in \mathbb N\setminus I$ and thus we have
\begin{equation}\label{Grig1}
\rho (G_i,X_i)=\frac{2|X_i|-1}{2|X_i|\omega(G_i,X_i)} + \frac{\omega(G_i,X_i)}{2|X_i|} \;\; {\rm if} \;\; i\in \mathbb N\setminus I.
\end{equation}
Assume that condition (b) holds. As we already explained, $|X_i|\to \infty$ as $i\to \infty$ in this case. Further, using (\ref{Grig1}) and the obvious equality $\omega(G_i,X_i)=1$ for all $i\in I$, we obtain $\lim\limits_{i\to \infty} \omega (G_i, X_i)/|X_i|=0$. Thus (b) $\Longrightarrow$ (d).
The converse implication follows from (\ref{Grig1}) and the well-known formulas: $\omega (G_i,X_i)\ge \sqrt{2|X_i| -1}$ for $i\in \mathbb N\setminus I$ \cite{Grig} and $\rho (G_i,X_i)=\sqrt{2|X_i|-1}/|X_i|$ for $i\in I$ \cite{Kes}.
\end{proof}

\begin{defn}\label{isr-def}
Let $\mathcal G=\{(G_i,X_i)\}_{i\in \mathbb N}$ be a sequence of groups $G_i$ with finite generating sets $X_i$. We say that $\mathcal G$ has \emph{infinitesimal spectral radius} if
any of the equivalent conditions from Lemma \ref{isr-char} holds.
\end{defn}

\begin{ex}\label{ex:free}
It is easy to see that a sequence $\{(F_i,X_i)\}$ of free groups $F_i$ and their bases $X_i$ has infinitesimal spectral radius if $|X_i|\to \infty $.
\end{ex}

We now discuss some examples without non-cyclic free subgroups. The lemma below may be thought of as a generalization of Example \ref{ex:free}.

\begin{lem}
Let $\mathfrak V$ be a variety of groups that contains a non-amenable group. Let $\mathcal G=\{(G_i, X_i)\}$ be a sequence of free groups $G_i$ in $\mathfrak V$ and their bases $X_i$. Assume that the cardinality of $X_i$ (strongly) increases as $i\to \infty$. Then $\mathcal G$ has infinitesimal spectral radius.
\end{lem}
\begin{proof}
The sequence $\mathcal G$ can be included in a bigger sequence $\mathcal H=\{(H_i, Y_i)\}$ of free groups $H_i$ in $\mathfrak V$ and their bases $Y_i$, where $|Y_i|=i$.  Since the property of having infinitesimal spectral radius is preserved by passing to subsequences, it suffices to prove the lemma for $\mathcal H$. Thus we can assume that $|X_i|=i$ without loss of generality.

Recall that the class of amenable groups is closed under quotients and direct limits. Thus if every $G_i$ is amenable, then every group in $\mathfrak V$ is amenable. Hence $G_i$ must be non-amenable for some $i$. By the Kesten criterion \cite{Kes}, we have $\left\| A_{X_i^{\pm 1}}\right\| <1$. Therefore, $\left\| A_{X_i^{\pm 1}}^n\right\|\to 0$ as $n\to \infty$.

Clearly $A_{X_i^{\pm 1}}^n=\lambda _G(s_n)$, where $s_n\in \mathbb ZG$ is a sum of $(2i)^n$ elements of $G_i$ (not necessarily pairwise distinct). Since $G_{(2i)^n}$ is free of rank $(2i)^n$ in $\mathfrak V$, there is a homomorphism $G_{(2i)^n}\to G_i$ whose extention to the corresponding group rings maps the sum $\sum\limits_{x\in X_{(2i)^n}} x$ to $s_n$. By part (b) of Lemma \ref{norms}, we have $\left\| A_{X_{(2i)^n}}\right\|\le \left\| A_{X_i^{\pm 1}}^n\right\|\to 0$ as $n\to \infty$. Therefore the sequence $\{ \left\| A_{X_i}\right\|\} $ contains a subsequence converging to $0$.

Let $k\ge l$ and let $X_k=\{x_1, \ldots, x_{k}\}$ and $X_l=\{ y_1, \ldots , y_{l}\}$ be bases in $G_k$ and $G_l$. Let $k=lq+r$, where $q\in \mathbb N$, $r\in \mathbb N\cup\{0\}$, and $0\le r<l$. For every $i\in \{1, \ldots, k\}$ there exists unique $j\in \{ 1, \ldots , l\}$ such that $(i-1)\equiv (j-1)(mod\, l)$; we denote this $j$ by $j(i)$. Since $X_k$ is a basis in $G_k$, the map defined by
$$
x_{i}\mapsto \left\{\begin{array}{ll}
                      y_{j(i)}, & {\rm if}\; 1\le i\le lq \\
                      1, & {\rm if}\; lq<i\le k
                    \end{array}\right.
$$  extends to a homomorphism $G_k\to G_l$ and hence to a homomorphism $\e\colon \mathbb CG_k\to \mathbb CG_l$. Note that $$\e ( x_1 +\cdots + x_{k}) = q y_1 +\cdots + q y_{l} + r1.$$ Applying successively part (b) of Lemma \ref{norms} and the triangle inequality, we obtain
$$
\begin{array} {rcl}
\left\| A_{X_k}\right\| & = & \frac1k \left\| \lambda _{G_k}(x_1 +\cdots + x_{k})\right\| \\ &&\\
& \le & \frac1k \left\| \lambda _{G_l}(q y_1 +\cdots + q y_{l} + r1)\right\|  \\ &&\\
& \le & \frac{q}{k} \left\| \lambda _{G_l}(y_1 +\cdots + y_{l} )\right\| + \frac{r}k  \\ &&\\
& < & \frac1l \left\| \lambda _{G_l}(y_1 +\cdots + y_{l} )\right\| +\frac{l}k  \\ &&\\
& = & \left\| A_{X_l}\right\| + \frac{l}k.
\end{array}
$$
The inequality $\left\| A_{X_k}\right\|< \left\| A_{X_l}\right\| + \frac{l}k$ for all $k\ge l$  together with the existence of a subsequence of $\{ \left\| A_{X_i}\right\|\} $ converging to $0$ obviously implies $\left\| A_{X_i}\right\|\to 0$ as $i\to \infty$.
\end{proof}

\begin{rem}
Note that the condition that $|X_i|$ strongly increases as $i\to \infty$ can be replaced by $|X_i|\to \infty $ as $i\to \infty$.
\end{rem}

\begin{cor}\label{bisr}
Let $G_i=B(i,n)$ be free Burnside groups of finite ranks $i\to \infty $ and odd exponent $n\ge 665$ and let $X_i$ be the standard (free) generating sets of $B(i,n)$. Then $\{(G_i, X_i)\}$ has infinitesimal spectral radius.
\end{cor}

\begin{proof}
Adyan \cite{A} proved that $B(m,n)$ is non-amenable for $m\ge 2$ and any odd $n\ge 665$. Thus the corollary follows immediately from the previous lemma . Alternatively, it can be derived directly from results of \cite{A}. Indeed in the course of proving the main theorem, Adyan showed that $\omega(B(i ,n), X_i) \le (2i-1)^{2/3}$ for any $i\ge 2$ and odd $n\ge 665$. Thus condition (d) from Lemma \ref{isr-char} holds.
\end{proof}

In the next lemma we denote by $\beta_1^{(2)}(G)$ the first $\ell^2$-Betti number of a group $G$.

\begin{lem}\label{b1}
Let $\mathcal G= \{(G_i, X_i)\}$ be a sequence of groups and generating sets such that $\beta_1^{(2)} (G_i)/ |X_i|\to 1 $ as $i\to \infty$. Then $\mathcal G$ has infinitesimal spectral radius.
\end{lem}

\begin{proof}
It is proved in \cite{LPV} that for every group $G$ generated by a finite set $X$, one has $h(G,X)\ge 2\beta_1^{(2)}(G,X)$. Note also that $h(G,X) \le 2|X|$ by the definition of the Cheeger constant. Hence $\mathcal G$ satisfies condition (c) from Lemma \ref{isr-char}.
\end{proof}

\begin{ex}\label{b1-ex}
There exist sequences of torsion groups satisfying assumptions of Lemma \ref{b1}. Indeed, for every $n\in \mathbb N$ and every $\e>0$, there exists an $n$-generated torsion group $G$  such that $\beta_1^{(2)} (G)\ge n-1 -\e$. The first such examples were constructed in \cite{Osi09}; variants of this construction leading to groups with some additional properties can be found in \cite{LO} and \cite{OT}.
\end{ex}

Recall that the \emph{girth} of a group $G$ with respect to a generating set $X$, denoted ${\rm girth}(G,X)$, is the length of the shortest non-trivial element in the kernel of the natural homomorphism $F(X)\to G$; here $F(X)$ is the free group with basis $X$, and ``length" means the word length in $F(X)$ with respect to $X$.

The next result will be used in Section 4 to prove Theorem \ref{main2}. Note that Example \ref{b1-ex} allows us to make the proof independent of the complicated Novikov-Adyan technique used in \cite{A}.

\begin{prop}\label{Pij}
There exists a collection $\{ (P_{ij}, X_{ij})\mid (i,j)\in \mathbb N\times \mathbb N\}$, where $P_{ij}$ is a group with a generating set $X_{ij}$, such that the following conditions hold.
\begin{enumerate}
\item[(a)] For any $i, j\in \mathbb N$, we have $|X_{ij}|= i$.
\item[(b)] For any $i\in \mathbb N$, ${\rm girth}(P_{ij},X_{ij})\to \infty $ as $j\to \infty$.
\item[(c)] For any map $j\colon \mathbb N\to \mathbb N$, the sequence $\{ (P_{ij(i)}, X_{ij(i)})\}_{i\in \mathbb N}$ has infinitesimal spectral radius.
\item[(d)] For any $i, j\in \mathbb N$, $P_{ij}$ has no non-cyclic free subgroups.
\end{enumerate}
Moreover, there exist such collections consisting entirely of torsion groups as well as collections consisting entirely of torsion free groups.
\end{prop}

\begin{proof}
Let $\{(G_i, X_i)\}$ be a sequence of torsion groups and generating sets satisfying (a) with infinitesimal spectral radius. For example, by Corollary \ref{bisr} we can take $G_i=B(i, 665)$ and let $X_i$ be the standard basis. Alternatively we can use groups from Example \ref{b1-ex}.

Let $G_i=F(X_i)/N_i$, where $F(X_i)$ is free with basis $X_i$. To construct a torsion free collection, we define $P_{ij}=F(X_i)/N_i^{(j)}$, where $N_i^{(j)}$ is the $j$th term of the derived series of $N_i$ (i.e., $N_i^{(1)}=[N_i, N_i]$ and $N_i^{(j+1)}=[N_i^{(j)}, N_i^{(j)}]$ for $j\in \mathbb N$). We also let $X_{ij}$ be the natural image of $X_i$ in $P_{ij}$.

It is well-known that for any normal subgroup $N$ in a free group $F$, the quotient group $F/[N,N]$ is torsion free (see, for example, \cite[Theorem 2]{Hig}). Hence all groups $P_{ij}$ are torsion free. Further, by a result of Levi (see \cite[Lemma 21.61]{N}), a strictly decreasing sequence of subgroups in a free group has trivial intersection provided every term of the sequence is a verbal subgroup of the previous term. Hence we have $\bigcap_{j=1}^\infty N_i^{(j)}=1$ for every $i$; this yields (b).

To prove (c) we note that for every $i,j\in \mathbb N$, there is a homomorphism $P_{ij}\to G_i$ that maps $X_{ij}$ to $X_{i}$. Hence $\left\| A_{X_{ij}}\right\|\le  \left\| A_{Xi}\right\|$ by Lemma \ref{norms} (b). Recall that the sequence $\{ (G_i,X_i)\}$ has infinitesimal spectral radius. Hence (c) follows from the characterization of sequences with infinitesimal spectral radius provided by Lemma \ref{isr-char} (a).

Finally, it is clear that $P_{ij}$ does not contain non-cyclic free subgroups as it is solvable-by-torsion. This completes the proof of the proposition in the torsion free case.

To construct a collection consisting of torsion groups, we again let $P_{ij}=F(X_i)/N_i^{(j)}$, but define $N_i^{(j)}$ by the rule $N_i^{(1)}=N_i$ and $N_i^{(j+1)}=(N_i^{(j)})^2$ for $j\in \mathbb N$. Then the same arguments as above yield (b)--(d).
\end{proof}


\section{Free Burnside groups}


\subsection{Preliminary information}
Our proof of Theorem \ref{main1} makes use the technique from \cite{book}. We recall main definitions here.

Given an alphabet $\mathcal A$, we denote by $|W|$ the length of a
word $W$ over $\mathcal A$. For two words $U,V$ over $\mathcal A$
we write $U\equiv V$ to express letter--by--letter equality.
If $\mathcal A$ is a generating set of a group $G$, we
write $U=V$ whenever two words $U$ and $V$ over $\mathcal A ^{\pm
1}$ represent the same element of $G$; we identify the words over
$\mathcal A ^{\pm 1} $ and the elements of $G$ represented by them.

The free Burnside group $B(m,n)$ of exponent $n$ and
rank $m$, is the free group in the variety defined by the identity
$X^n=1$. Throughout this section we assume that $n$ is odd and
large enough, and $m$ is a cardinal number greater than $1$. We stress that $m$ is not assumed to be finite, so all results of this section apply to groups of any cardinality.

By \cite[Theorem 19.1]{book}, the group $B(m,n)$ can be defined by the  presentation
\begin{equation}
B(\mathcal A, n)=\left\langle \mathcal A  \; \left| \; R=1, R\in \bigcup\limits_{i=1}^\infty
\mathcal R_i\right.\right\rangle ,\label{B}
\end{equation}
where $\mathcal A$ is an alphabet of cardinality $m$ and the sets of
relators $\mathcal R_0\subseteq \mathcal R_1\subseteq \ldots$  are constructed as follows.
Let $\mathcal R_0 =\emptyset $. By induction, suppose that we
have already defined the set of relations $\mathcal R_{i-1}$,
$i\ge 1$ and the sets of periods of ranks $1,\dots,i-1$. Let
$$
G(i-1)=\langle \mathcal A \mid R=1, R\in \mathcal R_{i-1}\rangle .
$$

For $i\ge 1$, a word $X$ in the alphabet $\mathcal A ^{\pm 1} $ is called {\it
simple in rank} $i-1$, if it is not conjugate to a power of a
shorter word or to a power of a period of rank $\le i-1$ in the group $G(i-1)$ (i.e., {\it in rank} $i-1$). We
denote by $\mathcal X_i$ a maximal subset of words
satisfying the following conditions.

\begin{enumerate}
\item[1)] $\mathcal X_i$ consists of words of length $i$ which are simple
in rank $i-1$.

\item[2)] If $A, B\in \mathcal X_i$ and $A\not\equiv B$, then $A$ is not
conjugate to $B$ or $B^{-1}$ in the group $G(i-1)$.
\end{enumerate}

\noindent Each word from $\mathcal X_i$ is called a {\it period of
rank $i$.} Let $\mathcal S _i=\{ A^n \mid A\in \mathcal X _i\} $. Then the set $\mathcal R _i$ is defined by $\mathcal R _i=\mathcal R _{i-1} \cup \mathcal S _i$.

Results from Chapters 5 and 6 of \cite{book} involve a sequence of
fixed positive small parameters
\begin{equation}\label{para}
\alpha, \; \beta,\; \gamma,\; \delta, \;\e,\; \zeta,\;  \iota=1/n.
\end{equation}
The exact relations between the parameters are described by a system of
inequalities, which can be made consistent by choosing each parameter in
this sequence to be sufficiently small as compared to all previous
parameters. In \cite{book}, this way of ensuring consistency is
referred to as the \emph{lowest parameter principle}
(see \cite[Section 15.1]{book}). For brevity, we also use
$\bar\alpha=\alpha+ 1/2$, $\bar\beta= 1-\beta$ and $\bar\gamma=1-\gamma$.

Recall that a van Kampen {\it diagram} $\Delta $ over an alphabet $\cal A$
is a finite, oriented, connected and simply-connected, planar 2-complex endowed with a
labeling function $\Lab : E(\Delta )\to {\cal A}^{\pm 1}\cup\{1\}$, where $E(\Delta
) $ denotes the set of oriented edges of $\Delta $, such that $\Lab
(e^{-1})\equiv \Lab (e)^{-1}$. Given a cell (that is a 2-cell) $\Pi $ of $\Delta $,
we denote by $\partial \Pi$ the boundary of $\Pi $; similarly,
$\partial \Delta $ denotes the boundary of $\Delta $.
An additional requirement for a diagram over a presentation $\langle {\cal A}\; | \; \mathcal R\rangle$ (or just over the group  given by this presentation) is that the boundary label of any cell $\Pi $ of $\Delta $ is equal to a cyclic permutation of a word $R^{\pm 1}$, where $R\in \mathcal R$. Given an edge $e$  or a path $p$ in a van Kampen diagram, we denote by $e_-$ and $e_+$ (respectively, by $p_-$ and $p_+$) its initial and terminal vertices.
Subpaths of $\partial \Delta$ are also called \emph{sections}.

All the diagrams under consideration are {\it graded}: by definition, a diagram $\Delta$ over $G(i)$ is called a diagram of \emph{rank $i$}  and a cell $\Pi$ of $\Delta$ labeled by a word from ${\cal S}_j$ has \emph{rank $j$} (written $r(\Pi)=j$).

The diagrams from  \cite{book} can also have $0$-edges labeled by $1$. The edges labeled
by the letters from ${\cal A}^{\pm 1}$ (i.e. ${\cal A}$-{\it edges}) have length $1$ and
every $0$-edge has length $0$. The length $|p|$ of every path $p=e_1\dots e_s$ in a diagram is the sum of the lengths $|e_i|$ of its edges. We also admit cells of rank $0$ (or $0$-cells). By definition, the boundary path of a $0$-cell either entirely consists
of $0$-edges or can have $0$-edges and exactly two $\cal A$-edges with mutually
inverse labels. In both cases the boundary label of a $0$-cell is freely equal to $1$.
The boundary labels of the cells of positive ranks are the defining relators from $\cal R$,
and these cells are called $\cal R$-cells. Note that the perimeter $|\partial \Pi|$ of an $\cal R$-cell $\Pi$ of rank $j\ge 1$ is equal to $nj.$

A $0$-{\it bond} between two cells $\Pi_1$ and $\Pi_2$ of positive ranks is a subdiagram $\Gamma$ of rank $0$ with a partition of its boundary path of the form $p_1q_1p_2q_2$,
where $q_1$ (respectively, $q_2$) is an $\cal A$-edge of the boundary $\partial\Pi_1$ (of $\partial\Pi_2$) and $|p_1|=|p_2|=0$. Similarly one defines a $0$-bond between a cell and a section $q$  or between two sections of $\partial\Delta$ .

For $j\ge 1$, a pair of distinct cells $\Pi_1$ and $\Pi_2$ of rank
$j$ of a diagram $\Delta$ is said to be a $j$-{\it
pair}, if their counterclockwise contours $p_1$ and $p_2$ are
labeled by $A^n$ and $A^{-n}$ for a period $A$ of rank $j$ and
there is a path $t$ from $(p_1)_-$ to $(p_2)_-$ without self-intersections
such that $\Lab(t)$ is equal to $1$ in rank $j-1$ (i.e., in $G(j-1)$). Then
the subdiagram with contour $p_1tp_2t^{-1}$ has label equal to $1$ in rank $j-1$ and so it
can be replaced in $\Delta$ by a diagram of rank $j-1$.

This surgery lexicographically decreases the {\it type} $\tau(\Delta)=(\tau_1,\tau_2,\dots)$ of the diagram, where $\tau_k$ is the number of the cells of rank $k$ in $\Delta.$
As result, we obtain a diagram with the same boundary label as
$\Delta$ but having no $j$-pairs for $j=1,2,\dots$. Such a diagram is called
{\it reduced}.

A similar transformation can be performed for a cell $\Pi_1$ and a section $q$ of $\partial\Delta$. Namely, let $A^{\pm 1}$ denote the period of rank $j$ corresponding to the cell $\Pi_1$ and let $q$ be a section of $\partial\Delta$ with a $A^{\pm 1}$-periodic label. (A word is called \emph{$B$-periodic} if it is a subword of a power of $B$.) Suppose that $q=q_1q_2$, where the word $\Lab(q_1)$  (the word $\Lab(q_2)$) is a subword of a power of $A$ ending with $A^{\pm 1}$ (respectively, starting with $A^{\pm 1}$). As above, let $t$ be a path without self-intersections such that $t_-=(p_1)_-$, $t_+=(q_1)_+=(q_2)_-$, and $\Lab (t)$ equals $1$ in rank $j-1$. Then the cell $\Pi_1$ is called {\it compatible} with $q$. Note that $\Lab(q_1t^{-1}p_1tq_2)$ is equal in rank $j-1$ to the $A^{\pm 1}$-periodic word $\Lab(q_1)A^{\pm n}\Lab(q_2)$.  Therefore one can decrease the type of $\Delta$  replacing $\Pi_1$
by a subdiagram of rank $\le j-1$ and replacing the section $q$ by another section
with another $A^{\pm 1}$-periodic label.

The inductive definition of a {\it contiguity subdiagram} $\Gamma$ of an $\cal R$-cell
$\Pi_1$ to an $\cal R$-cell $\Pi_2$ (or to a section $q$ of $\partial\Delta$) depends
on the parameters (\ref{para}) and takes 3 pages in \cite{book}. Since we do not use
the details of that definition here, it suffices to say that $\Gamma$ is given with the partition $p_1q_1p_2q_2$ of its boundary $\partial\Gamma$, where $q_1$ and $q_2$ are some
sections of $\partial\Pi_1$ and $\partial\Pi_2$ (or, respectively, of $q$), and $\Gamma$
contains neither $\Pi_1$ nor $\Pi_2$. We write $\partial(\Pi_1,\Gamma,\Pi_2)=p_1q_1p_2q_2$
(respectively, $\partial(\Pi_1,\Gamma,q)=p_1q_1p_2q_2$) to distinguish the {\it contiguity arcs} $q_1$ and $q_2$ of $\Gamma$ and its {\it side arcs} $p_1$ and $p_2$.

The ratio $|q_1|/|\partial\Pi_1|$ (the ratio $|q_1|/|q|$) is called the
{\it degree of contiguity} of the cell $\Pi_1$ to $\Pi_2$ (to the section $q$).
It is denoted by $(\Pi_1,\Gamma,\Pi_2)$ (respectively, $(\Pi_1,\Gamma,q)$).

 A path $p$ in $\Delta$ is called {\it geodesic} if $|p|\le |p'|$ for any path $p'$
with the same endpoints. A diagram $\Delta$ is said to be an \emph{$A$-map
if} (1) any subpath of length $\le \max(j,2)$ of the contour of
an arbitrary cell of rank $j\ge 1$  is geodesic in
$\Delta$ and (2) if
$\pi, \Pi$ are $\cal R$-cells and $\Gamma$ is a diagram of contiguity
of $\pi$ to $\Pi$ with standard contour $p_1q_1p_2q_2$, where $|q_1|\ge \varepsilon|\partial\pi|$, then $|q_2|<(1+\gamma)r(\Pi)$. We note that by \cite[Lemma 19.4]{book}, every reduced diagram $\Delta$ over the presentation (\ref{B}) is an $A$-map and hence we can apply all lemmas formulated for $A$-maps in Chapter 5 of \cite{book} to reduced diagrams.

 A section $q$ of $\partial\Delta$ is called a
{\it smooth section of rank} $k > 0$ (we write
$r(q)=k$) if (1) every subpath of $q$ of length $\le \max(k,2)$  is
geodesic in $\Delta$ and
(2) for each contiguity subdiagram $\Gamma$ of a cell $\pi$ to $q$ with
$\partial(\pi,\Gamma,q)=p_1q_1p_2q_2$ and $(\pi, \Gamma,q)\ge \varepsilon$, we have $ |q_2|<(1+\gamma)k$. Note that by \cite[Lemma 19.5]{book}, if the label of a section $q$ of the boundary in a reduced diagram $\Delta$ is an $A$-periodic word, where $A^{\pm 1}$
is a period of some rank $j\ge 1$, and $\Delta$ has no cells compatible with $q$,
then $q$ is a smooth section of rank $j$ in $\partial\Delta$.

The following can be easily derived from the definitions (or see \cite[Lemma 15.1]{book}).

\begin{lem}\label{151} Let $\Pi$ be a cell of rank $k\ge 1$ in a reduced diagram $\Delta$.
If a subpath $q$ of $\partial\Pi$ is a section of the boundary of a subdiagram $\Gamma$ of $\Delta$, and $\Gamma$ does not contain $\Pi$, then $q$ is a smooth section of rank $k$ in $\partial\Gamma$.
\end{lem}

The main property
of smooth sections is that they are almost geodesic.

\begin{lem}[{\cite[Theorem 17.1]{book}}]\label{smoo}If $qt$ is the boundary path of a reduced diagram $\Delta$, where the section $q$ is smooth, then $\bar\beta|q|\le |t|$
\end{lem}

Another assertion with similar proof is the following.

\begin{lem}[{\cite[Corollary 17.1]{book}}]\label{beta} If a reduced diagram $\Delta$ contains an $\cal R$-cell $\Pi$, then $|\partial\Delta|>\bar\beta|\partial\Pi|$.
\end{lem}

We also need some (simplified versions of)  properties of contiguity subdiagrams proved in \cite{book}.

\begin{lem}\label{cont}
Let $\Delta$ be a reduced diagram
and $\Gamma$ a contiguity subdiagram of a cell $\pi$ to
a cell $\Pi$ (or to a section $q$ of $\partial\Delta$), let
$\partial(\pi,\Gamma,\Pi)=p_1q_1p_2q_2$ (respectively, $\partial(\pi,\Gamma, q)=p_1q_1p_2q_2$)
 and $\psi=(\pi,\Gamma, \Pi)$ (respectively, $\psi=(\pi,\Gamma, q)$). Then the following conditions hold:
\begin{enumerate}
\item[(a)] $\max (|p_1|,|p_2|)<\zeta n r(\pi)=\zeta|\partial\pi|$ \cite[Lemma 15.3]{book};

\item[(b)] if $\psi\ge \varepsilon$, then $|q_1|<(1+2\beta)|q_2|$ \cite[Lemma 15.4]{book};

\item[(c)] if $\psi\ge \varepsilon$ and $q$ is a smooth section, then $|q_1|>(1-2\beta)|q_2|$ \cite[Lemma 15.4]{book};

\item[(d)] if $\psi\ge \varepsilon$, then $|\partial\pi|<\zeta|\partial\Pi|$
(or $r(\pi)<\zeta r(q)$ if $q$ is smooth) \cite[Lemma 15.5]{book};

\item[(e)] $\psi<\bar\alpha$ if $q$ is smooth \cite[Lemma 15.8]{book}.
\end{enumerate}
\end{lem}

\begin{lem}[{\cite[Corollary 16.1]{book}}]\label{gamma}
If $q^1\dots q^l$ ($\l\le 4$) is the boundary path of a reduced diagram $\Delta$
having $\cal R$-cells, then
$\Delta$ has an $\cal R$-cell $\Pi$ and disjoint contiguity subdiagrams $\Gamma_1,\dots, \Gamma_l$ of $\Pi$
to $q^1,\dots,q^l$, respectively (some of them may be absent), with
$\sum_{k=1}^l(\Pi,\Gamma_k,q^k)>\bar\gamma=1-\gamma$ (which is close to $1$).

Such a cell $\Pi$ is called a {\rm $\gamma$-cell} of $\Delta$.
\end{lem}

\subsection{Repelling sections and towers in van Kampen diagrams}\label{trs}

By default, all diagrams discussed in this section are over the presentation (\ref{B}).

\begin{defn} \label{c} Let $q$ be a section of the boundary $\partial\Delta$ of a reduced diagram $\Delta$. We say
that an $\cal R$-cell $\Pi$ of $\Delta$ is $c$-close to $q$ if there is a subdiagram $\Gamma$ with a
contour $s_1t_1s_2t_2$ such that $\Gamma$ does not contain $\Pi$, $t_1$ is a subpath of $\partial\Pi$
of length at least $c |\partial\Pi|$, $t_2$ is a subpath of $q$, and $|s_1|, |s_2| <\zeta |\partial\Pi|$.
\end{defn}

\begin{rem} \label{close} Note that if there is a contiguity subdiagram $\Gamma$ of $\Pi$ to $q$ with $(\Pi,\Gamma,q)\ge c$,
then $\Pi$ is $c$-close to $q$ by Lemma \ref{cont} (a).
\end{rem}

\begin{defn} A section $q$ of $\partial\Delta$ whose label is a reduced word is called
$c$-repelling  if the reduced diagram $\Delta$ has no $\cal R$-cells $c$-close to $q$.

Below we say that $q$ is {\it repelling} if it is $(1-\alpha)$-repelling.
\end{defn}

The following lemma is a modification of \cite[Lemma 17.2]{book}; the smoothness
of the section $q_1$ is replaced here by the repelling condition.

\begin{lem} \label{rep} Let $\Delta$ be a reduced diagram with contour $p_1q_1p_2q_2$,
where $q_1$ is a repelling section, $q_2$ is a smooth one, $|q_2|>0$, and $|p_1|+|p_2|\le\gamma |q_2|$.
Then either there is a $0$-bond between $q_1$ and $q_2$ or there is a cell $\Pi$ in $\Delta$
and disjoint contiguity subdiagrams $\Gamma_1$ and $\Gamma_2$ of $\Pi$ to $q_1$ and $q_2$, respectively,
such that $(\Pi,\Gamma_1,q_1)+ (\Pi,\Gamma_2,q_2)>\bar\beta=1-\beta$ (we call $\Pi$ a {\it $\beta$-cell} ).
\end{lem}
\proof If $r(\Delta)=0$, then there is a $0$-bond between $q_1$ and $q_2$ because $|p_1|+|p_2|<|q_2|$
and there are no $0$-bonds connecting $q_2$ to itself by Lemma \ref{smoo} with $|t|=0$. Hence proving
by contradiction, we may assume that $r(\Delta)>0$ and $\Delta$ is a counter-example with minimal number of $\cal R$-cells denoted
$|\Delta(2)|$.

By Lemma \ref{gamma}, there is a $\gamma$-cell $\Pi$ in $\Delta$. By Lemma \ref{cont} (e), its degree of contiguity to $q_2$ is less than $\bar\alpha=1/2+\alpha$. Since the section $q_1$ is
repelling and $\bar\gamma=1-\gamma>1-\alpha$, it it easy to see from  Remark \ref{close}, that $\Pi$ satisfies one of the following four conditions, as in \cite[Lemma 17.2]{book}.

{\bf 1.} The degree of contiguity of $\Pi$ to $p_1$ or to $p_2$ is greater than $\bar\alpha$.

{\bf 2.} The sum of the degrees of contiguity of $\Pi$ to $p_1$ (or to $p_2$) and to $q_1$ (or to $q_2$)
is greater than $\bar\gamma$.

{\bf 3.} There are disjoint contiguity subdiagrams $\Gamma$, $\Gamma_1$, and $\Gamma_2$ of $\Pi$
to $p_1$ (or to $p_2$), $q_1$, and $q_2$, resp., with the sum of degrees $>\bar \gamma$ and
with $(\Pi,\Gamma,p_1)>\beta-\gamma$.

{\bf 4.} There are disjoint contiguity subdiagrams $\Gamma_1$ and $\Gamma_2$ of $\Pi$ to
$p_1$ and $p_2$, respectively, such that $(\Pi,\Gamma_1,p_1)+ (\Pi,\Gamma_2,p_2)>\beta-\gamma$.

\begin{figure}
 \centering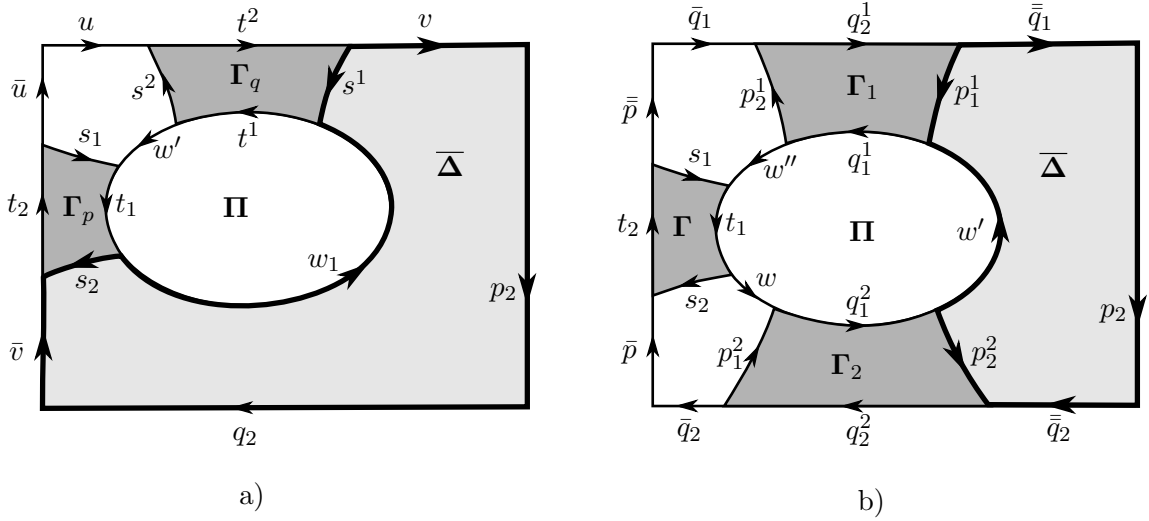\\
 \caption{Cases 2 and 3 in the proof of Lemma \ref{rep}}\label{figAB}
\end{figure}

{\bf Case 1.} This case is eliminated exactly as Case 1) in the proof of \cite[Lemma 17.2]{book} since
the condition on $q_1$ is not used there.

{\bf Case 2.} Without loss of generality we can assume that the sum of contiguity degrees of $\Pi $ to $p_1$ and one of $q_1$, $q_2$ is greater than $\bar \gamma$. Let $\Gamma_p$ ($\Gamma_q$) be the contiguity subdiagram of $\Pi$ to $p_1$
(respectively, to $q_1$ or to $q_2$). If it is a contiguity subdiagram to $q_2$, then we just repeat
the argument from Case 2) in the proof of \cite[Lemma 17.2]{book} since $q_2$ is smooth as in \cite{book}.
So we may assume that $\Gamma_q$ is a contiguity subdiagram of $\Pi$ to $q_1$.

We set $\partial(\Pi,\Gamma_p,p_1)=s_1t_1s_2t_2$,  $\partial(\Pi,\Gamma_q,q_1)=s^1t^1s^2t^2$,
write $\partial\Pi$ in the form $t_1w_1t^1w^1$, and factorize $q_1= ut^2v$,
$p_1=\bar vt_2\bar u$
(Fig. \ref{figAB} a)). Note that $|w_1|+|w^1|<\gamma |\partial\Pi|$.

Since the section $q_1$ is repelling, we have $(\Pi,\Gamma_q,q_1)\le 1-\alpha$ by Remark \ref{close}, whence $(\Pi,\Gamma_p,p_1)>\bar\gamma-(1-\alpha)>\alpha-\beta>\varepsilon,$ so by Lemma 1 \ref{cont} (b),

\begin{equation}\label{tt2} |t_2|> (1+2\beta)^{-1}(\alpha-\beta)|\partial\Pi|>3\alpha|\partial\Pi|/4
\end{equation}

We define $\bar p_1=\bar vs_2^{-1}w_1(s^1)^{-1}$. By Lemma \ref{cont} (a),
we have $|s_j|, |s^j| < \zeta|\partial\Pi|$. It therefore follows from (\ref{tt2}) that

\begin{equation}\label{barp1}
|p_1|-|\bar p_1|\ge |t_2|-|s_2|-|w_1|-|s^1|>(\frac34\alpha - 2\zeta -\gamma)|\partial\Pi|> 0
\end{equation}

By (\ref{barp1}), the hypothesis of the lemma holds for the subdiagram $\bar\Delta$ with the contour
$\bar p_1 v p_2 q_2$. Since $|\bar\Delta(2)|<|\Delta(2)|$ we come to a contradiction
with the minimality of $\Delta$.

{\bf Case 3.} We introduce the following notation: $\partial(\Pi,\Gamma,p_1)=s_1t_1s_2t_2$,
$\partial(\Pi,\Gamma_i, q_i)=p_1^iq_1^ip_2^iq_2^i$ ($i=1,2$), and $t_1wq_1^2w'q_1^1w''$ is
the contour of $\partial\Pi$. We also factorize $p_1=\bar p t_2\bar{\bar p}$, $q_1=\bar q_1 q_2^1 \bar{\bar q}_1$, and $q_2= \bar{\bar q}_2 q_2^2 \bar q_2$ (Fig. \ref{figAB} b)).

The path $p'= (p_2^2)^{-1}w'(p_1^1)^{-1}$ cuts up $\Delta$. By Lemma \ref{cont} (a), $|p'|<(2\zeta+\gamma)|\partial\Pi|.$ Since $\beta-\gamma>\varepsilon$, it follows from Lemma \ref{cont} (b) applied to $\Gamma$ that
$|t_2|>(1+2\beta)^{-1}|t_1|>(1-2\beta)(\beta-\gamma)|\partial\Pi|$. Therefore

\begin{equation}\label{p'} |p_1|-|p'|>|\bar p|+ ((1-2\beta)(\beta-\gamma)- 2\zeta-\gamma)|\partial\Pi|>|\bar p|+\frac{\beta}{2}|\partial\Pi|
\end{equation}

By Lemma \ref{smoo}, we have $|q_2^2\bar q_2|\le\bar\beta^{-1}|\bar ps_2^{-1}wq_1^2p_2^2|$, and using Lemma \ref{cont} (a), we obtain:

\begin{equation}\label{q2} |q_2^2\bar q_2|<\bar\beta^{-1}|\bar p|+\bar\beta^{-1}(1+2\zeta)|\partial\Pi|
\end{equation}
 It follows from (\ref{q2}) and (\ref{p'}) that
$|p_1|-|p'|>\gamma |q_2^2\bar q_2|=\gamma (|q_2|-|\bar{\bar q}_2|)$ because $\gamma \bar\beta^{-1}<1$ and $\gamma\bar\beta^{-1}(1+2\zeta)< \frac{\beta}{2}$. Together with the assumption of the lemma, this implies that
$|p'|+|p_2|<\gamma|\bar{\bar q}_2|$, and so the subdiagram $\Delta'$ with the contour $p'\bar{\bar q}_1p_2\bar{\bar q}_2$
is a smaller counter-example (since $|\Delta'(2)|<|\Delta(2)|$), a contradiction.

\begin{figure}
 \centering\hspace*{10mm}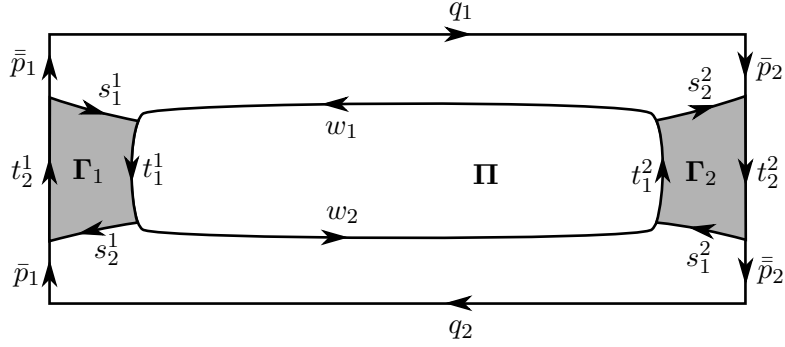\\
 \caption{Case 4 in the proof of Lemma \ref{rep}}\label{figC}
\end{figure}

{\bf Case 4.} We set $s_1^it_1^is_2^it_2^i=\partial(\Pi,\Gamma_i,p_i)$, $p_i=\bar p_it^i_2\bar{\bar p}_i$ ($i=1,2$),
and let $w_1t_1^1w_2t_1^2$ be the contour of $\Pi$ (Fig. \ref{figC}). Then exactly as in the proof of \cite[Lemma 17.2]{book} (see inequality (9) is Subsection 17.3), we have

\begin{equation}\label{9} |p_1|+|p_2|>|\bar p_1|+  |\bar{\bar p}_2|+(\bar\beta(\beta-\gamma)-4\zeta)|\partial\Pi|
\end{equation}

On the other hand, we compare $q_2$ with the homotopic path $\bar{\bar p}_2^{-1}s_1^2w_2^{-1}s_2^1\bar p_1^{-1}$
and obtain by   Lemmas \ref{smoo} and \ref{cont} (a):
\begin{equation}\label{10}
|q_2|<\bar\beta^{-1}(|\bar p_1|+  |\bar{\bar p}_2|+(1+2\zeta)|\partial\Pi|)
\end{equation}
Since $\gamma\bar\beta^{-1}(1+2\zeta)< \bar\beta(\beta-\gamma)-4\zeta$, the inequalities (\ref{9}) and (\ref{10})
imply that $\gamma|q_2|<|p_1|+|p_2|$ contrary to the assumption of the lemma.

The lemma is proved by contradiction.
\endproof

\begin{defn} Let $p$ be a section of the boundary of a cell $\Pi$ or of a diagram. Assume that there is a subdiagram $\Gamma$ of rank $0$ with a contour $p_1pp_2q'$, where $q'$ is a subpath
of $q$ and $|p_1|=|p_2|=0$
Then we say  that $p$ is {\it immediately close} to $q$.
\end{defn}

\begin{lem} \label{nobeta}
\begin{enumerate}
\item[(a)] Let $\Delta$ be a reduced diagram with contour $p_1q_1p_2q_2$,
where $q_1$ is a repelling section, $q_2$ is a smooth one, and $|p_1|+|p_2|<\gamma |q_2|$,
but $\Delta$ has no $\beta$-cells. Then there is a subpath $q$ of $q_2$ of length $>|q_2|-\gamma^{-1}(|p_1|+|p_2|)$
which is immediately close to $q_1$.

\item[(b)] Let $\Delta$ be a reduced diagram of positive rank with boundary contour $p$ having reduced label. Then there is an ${\cal R}$-cell $\Pi$ in $\Delta$ and a subpath $q$ of $\partial\Pi$ of length greater than $(1/2-\alpha-2\beta)|\partial\Pi|> 2/5|\partial\Pi|$ which is immediately close to $p$.

\item[(c)] If $p$ is a section of a reduced diagram
$\Delta$ and there is an $\cal R$-cell $\Pi$
with $(\Pi,\Gamma, p)> c$ for some $c\in(\alpha+\beta, 1/2-\alpha],$ then there is an ${\cal R}$-cell $\pi$ in $\Delta$ and a subpath $q$ of $\partial\pi$ of length $>(c-\alpha)|\partial\pi| $ which is immediately close to $p$.
\end{enumerate}
\end{lem}

\proof (a) By Lemma \ref{rep}, there is a $0$-bond between $q_2$ and $q_1$. Let $E$ be  such a bond closest to the
vertex $(q_2)_-$, i.e., let $e$ be the first $\cal A$-edge of $q_2$ connected by such $0$-bond with an edge of $q_1$.
The subpath $q'$ of $q_2$ going from $(q_2)_-$ to $e_-$ has length $< \gamma^{-1}|p_2|$ or $0$ since otherwise
by Lemma \ref{rep}, one could find a $0$-bond between $q'$ and $q_1$. Similarly we define the path $q''$ of length
$<\gamma^{-1}|p_1|$ such that $q_2=q'qq''$. Hence $|q|> |q_2|-\gamma^{-1}(|p_1|+|p_2|)>0$ and we have
a subdiagram $\Gamma$ with contour $qp_1'tp'_2$, where $|p'_1|=|p'_2|=0$ and $t$ is a subpath of $q_1$.
The diagram $\Gamma$ has rank $0$. Indeed, otherwise it should have a $\gamma$-cell by Lemma \ref{gamma}, and
since $\gamma<\beta$, that $\gamma$-cell has to be a $\beta$-cell in $\Delta$, which contradicts the assumption
of the lemma. So $\Gamma$ is the required contiguity subdiagram.

(b) By Lemma \ref{gamma}, there is a cell $\Pi$ in $\Delta$ and a contiguity subdiagram $\Delta_0$
such that $(\Pi,\Delta_0,q)>\bar\gamma>1/2-\alpha-\beta$. By Remark \ref{close}, $\Pi$ is $(1/2-\alpha-\beta)$-close
to $q$. Thus we may choose a cell $\Pi$ such that the subdiagram $\Gamma$ of its $(1/2-\alpha-\beta)$-closeness
with contour $s_1t_1s_2t_2$ (as in Definition \ref{c})
 has no $\cal R$-cells $\pi$ $(1/2-\alpha-\beta)$-close to $t_2$.
 This implies that
the section $t_2$ is repelling in $\Gamma$.  Furthermore, $\Gamma$ has no $\beta$-cells $\pi$ since the degree of contiguity
of $\pi$ to $\Pi$ is less than $1/2+\alpha$ by Lemma \ref{cont} (e), and so the degree of
contiguity of $\pi$ to $t_2$ should be greater than $(1-\beta)-(1/2+\alpha)=1/2-\alpha-\beta$, which is impossible
by the choice of $\Gamma$ and Remark \ref{close}. Hence we may apply the
statement (a) to $\Gamma$ taking into account that $|s_1|,|s_2|<\zeta|\partial\Pi| $ by Lemma \ref{cont} (a).
We obtain the required
subarc $q$ of $t_1$ of length  at least $$|t_2|-\gamma^{-1}(2\zeta |\partial\Pi|)> (1/2-\alpha-\beta -2\zeta\gamma^{-1})|\partial\Pi|>(1/2-\alpha-2\beta)|\partial\Pi|,$$
that is immediately close to $t_2$, as desired.

(c) It suffices to repeat the argument from (b) replacing the constant $1/2-\alpha$ by $c$.
\endproof

\begin{rem}\label{55} The claim (b) of Lemma \ref{nobeta} is a modification of Lemma 5.5 \cite{Ol82}.
\end{rem}

\begin{lem}\label{dlina} Let $q$ be a repelling section of a reduced diagram $\Delta$ with boundary
path $pq$. Then $|q|\le |p|^2$.
\end{lem}

\begin{figure}
 \centering\hspace*{10mm}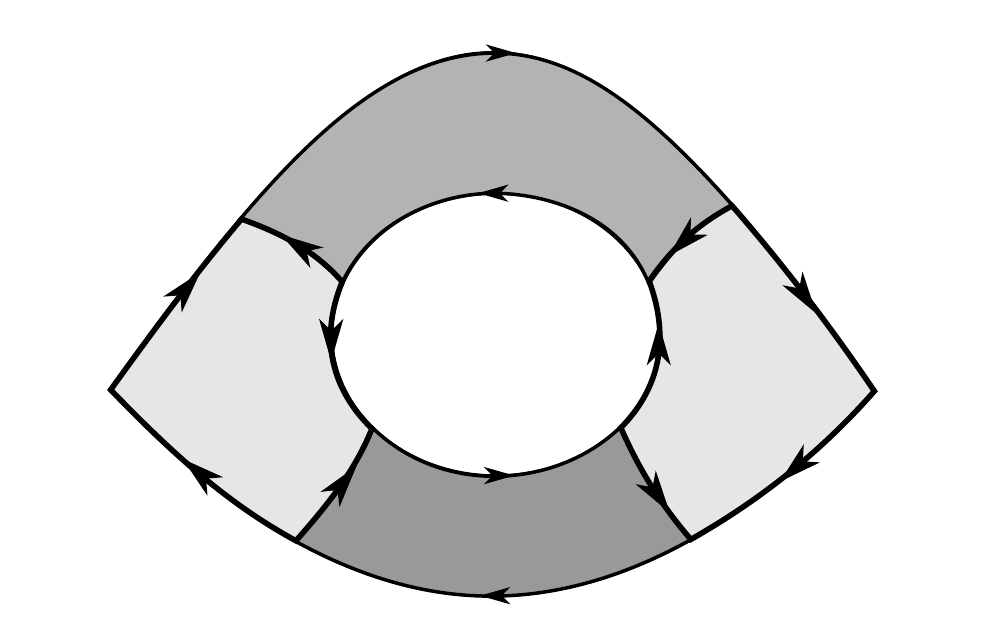\\
 \caption{}\label{figD}
\end{figure}

\proof Assume that $\Delta$ is a counter-example with minimal $|p|$. If $r(\Delta)=0$, then every edge of $q$ is connected by a $0$-bond with an edge of $p$ since $\Lab(q)$
is a reduced word. Hence $|q|\le |p|\le |p|^2,$ a contradiction. Therefore $r(\Delta)>0$, and so by Lemma \ref{gamma}, $\Delta$ has a $\gamma$-cell $\Pi$,
i.e., there are disjoint contiguity subdiagrams $\Gamma_p$ and $\Gamma_q$ of $\Pi$ to $p$ and $q$, respectively
(one of them may be absent) with $\psi_p+\psi_q >\bar\gamma,$
where $\psi_p = (\Pi,\Gamma_p,p)$, and $\psi_q = (\Pi,\Gamma_q,q)$. By Remark \ref{close}, $\psi_q < 1-\alpha$.
Also $\psi_p<\bar\alpha$ since otherwise $\Delta$ could be replaced by a subdiagram with boundary $\bar p q$,
where $|\bar p|<|p|$ (as this was shown in the proof of \cite[Theorem 17.1]{book}). Hence both $\Gamma_p$ and $\Gamma_q$ exist and

\begin{equation}\label{psi}
\alpha-\gamma < \psi_p <\bar\alpha\;\;\; and \;\;\;  \bar\gamma - \bar\alpha < \psi_q < 1-\alpha
\end{equation}

Let $\partial(\Pi,\Gamma_p,p)=s_1t_1s_2t_2$, $\partial(\Pi,\Gamma_q,q)=p_1q_1p_2q_2$
and suppose the contour of $\Pi$ is factorized as $q_1wt_1u$ (Fig. \ref{figD}). We denote $P=|\partial\Pi|$.
Then $|w|+|u|<\gamma P$ by the definition of $\gamma$-cell, and $\Pi$ gives two paths connecting $p$ and $q$ with
\begin{equation}\label{conn}
|p_1u^{-1}s_2|<(\gamma+2\zeta)P<2\gamma P\;\;\; and\;\;\;
|s_1w^{-1}p_2|<2\gamma P
\end{equation}
by Lemma \ref{cont} (a).

We factorize $p= p''t_2p'$ and $q=q'q_2q''$. Then $|t_2|>(1+2\beta)^{-1}|t_1|$ by Lemma \ref{cont} (b) and (\ref{psi}) since $\alpha-\gamma >\varepsilon$, and therefore
\begin{equation}\label{t2}
|t_2|>(1+2\beta)^{-1}(\alpha-\gamma)P>\frac 34 \alpha P
\end{equation}

  The path $s_1w^{-1}p_2$ is shorter than $t_2$ by (\ref{conn}, \ref{t2}), and so
$p_2^{-1}w s_1^{-1}p'$ is shorter than $p$. Hence the statement of the lemma holds for the subdiagram $\Delta'$ with the contour $(p_2^{-1}ws_1^{-1}p')q'$, and also for the subdiagram $\Delta''$ bounded by $(p''s_2^{-1}up_1^{-1})q''$, that is by (\ref{conn}),
\begin{equation}\label{shtrihi}
|q'|\le (|p'|+2\gamma P)^2 \;\;\; and\;\;\; |q''|\le (|p''|+2\gamma P)^2
\end{equation}

Now consider the subdiagram $\Gamma_q$. The section $q_1$ of it is smooth by Lemma \ref{151},
and $q_2$ is repelling in $\Gamma_q$ since $q$ is repelling in $\Delta$. By Lemma \ref{cont} (a)
and by (\ref{psi}), we have $|p_1|+|p_2|\le 2\zeta P< \gamma (\bar\gamma - \bar\alpha)P<\gamma |q_1|$.
Therefore we may apply Lemma \ref{rep} to $\Gamma_q$. It gives either a $0$-bond between $q_1$ and $q_2$
or a $\beta$-cell $\pi$ in $\Gamma_q$. Since the degree of contiguity of $\pi$ to $q_1$ is at least
$\bar\beta-(1-\alpha)=\alpha-\beta$ by Remark \ref{close}, and $\alpha-\beta>\varepsilon$, we have $|\partial\pi|<\zeta P$ by Lemma \ref{cont} (d).

As above for $\Pi$, the contiguity subdiagrams $\Gamma_1$ and $\Gamma_2$ of the $\beta$-cell $\pi$ to $q_1$ and $q_2$, respectively, give us two paths $x$ and $x'$
connecting  $q_1$ and $q_2$ of length $< (\beta+2\zeta)|\partial\pi|<2\beta\zeta P<\zeta P$ (Fig. \ref{figE} a)).
Thus using paths of length $<\zeta P$ one can cut $\Gamma_q$ in several subdiagrams $\Gamma^1,  \dots\Gamma^l,$
where $\partial\Gamma^i= x_i y_ix_{i+1}^{-1}z_i$ with $|x_i|<\zeta P$, $q_1=y_1\dots y_l$, $q_2=z_l\dots z_1$.
The cutting process stops only if $|y_i|\le \gamma^{-1}(|x_i|+|x_{i+1}|)< 2\zeta\gamma^{-1}P$ for all $i$-s.
Indeed, if $|y_i| > 2 \zeta\gamma^{-1} P$ and there are no $0$-bonds between $y_i$ and $q_2$, then again  there is  a cell $\pi'$ which is a $\beta$-cell of  $\Gamma^i$. Then the contiguity
subarc $v$ of $y_i$ to $\pi'$ could not be the entire $y_i$ since by Lemma 15.4 \ref{cont} (c),
$|v|<(1-2\beta)^{-1}|\partial\pi'|< (1-2\beta)^{-1}\zeta P< 2\zeta P <|y_i|;$
and so $\pi'$ provides   a further cut of $\Gamma^i$ decreasing $|y_i|$.

\begin{figure}
 \centering\hspace*{10mm}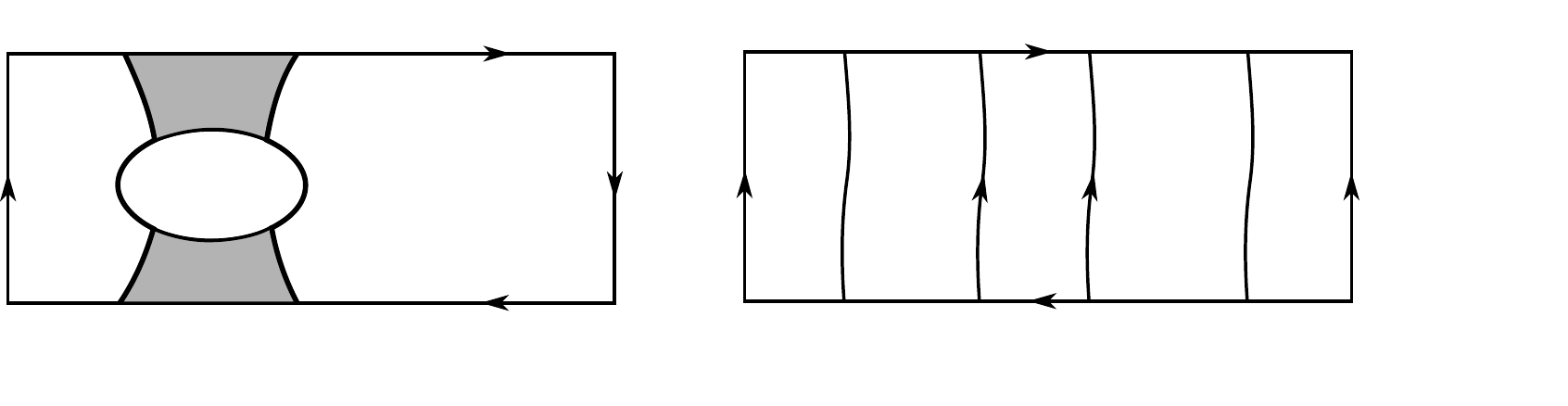\\
 \caption{}\label{figE}
\end{figure}

Combining several consecutive subdiagrams $\Gamma^i$, we obtain a rougher partition of the subdiagram $\Gamma_q$ into subdiagrams $\Delta_1,\dots,\Delta_m$,
with $\partial\Delta_j= f^jg^j_1(f^{j+1})^{-1}g^j_2,$ where $|f^j|<\zeta P$, $q_1=g_1^1\dots g_1^m$, $q_2=g_2^m\dots g_2^1$, and
\begin{equation}\label{rough}
2\zeta \gamma^{-1}P \le|g_1^j|<4\zeta \gamma^{-1}P \;\; (j=1,\dots,m)
\end{equation}
(Fig. \ref{figE} b)). Together with (\ref{t2}) this gives inequalities
\begin{equation}\label{f}
|f^jg^j_1(f^{j+1})^{-1}|<6\zeta\gamma^{-1} P<\frac34\alpha P<|t_2|<|p|
\end{equation}
Hence the statement of the lemma holds for
every $\Delta_j$, whence by (\ref{f}), $|g_2^j|\le (6\zeta\gamma^{-1} P)^2$. Note that by (\ref{rough}), $$m\le |q_1|/(2\zeta\gamma^{-1} P)\le P/(2\zeta\gamma^{-1} P)= \zeta^{-1}\gamma/2$$
Therefore
\begin{equation}\label{sum}
|q_2|=\sum_{j=1}^m|g_2^j|<m(6\zeta\gamma^{-1} P)^2\le (\zeta^{-1}\gamma/2)(6\zeta \gamma^{-1} P)^2=18\zeta \gamma^{-1}P^2
\end{equation}
Taking into account (\ref{shtrihi}) and (\ref{sum}), we obtain:
\begin{equation}\label{q}
|q|= |q'|+|q_2|+|q''|< (|p'|+2\gamma P)^2 + (|p''|+2\gamma P)^2 +18\zeta\gamma^{-1} P^2
\end{equation}

By inequality (\ref{t2}), the right-hand side of  (\ref{q}) does not exceed
$$(|p'|+ |p''|+4\gamma P +\sqrt{18\zeta\gamma^{-1}} P)^2 <(|p'|+|p''|+\frac34\alpha P)^2<(|p'|+|p''|+|t_2|)^2=|p|^2$$
The obtained inequality $|q|\le |p|^2$ proves the lemma by contradiction.
\endproof

We will use the following inductive definition for a {\it tower} of height $h\ge 0.$

\begin{defn} Let a path $p$ be labeled by a non-empty reduced word. Then it itself is a tower of
height $0$ with the base $p$. More accurately, this tower is a diagram of rank $0$
with the boundary $pq_0$, where $\Lab(q_0)\equiv \Lab(p)^{-1}$.

Assume a reduced diagram $\Delta_h$ has contour $pq_h$, where the labels of $p$ and $q_h$ are reduced words,
and assume that its  subdiagram $\Delta_{h-1}$ with a contour $pq_{h-1}$ is a tower of height $h-1\ge 0$
with base $p$. If the subdiagram $\Gamma_h$ with the contour $q_{h-1}^{-1}q_h$ has exactly one $\cal R$-cell
$\Pi_h$, and the degree of contiguity of $\Pi_h$ to $q_{h-1}$ in $\Gamma_h$ is at least $2\alpha$, then
$\Delta_h$ is a tower of height $h$ with the base $p$ (Fig. \ref{figF} a)).
\end{defn}

\begin{figure}
 \centering\hspace*{10mm}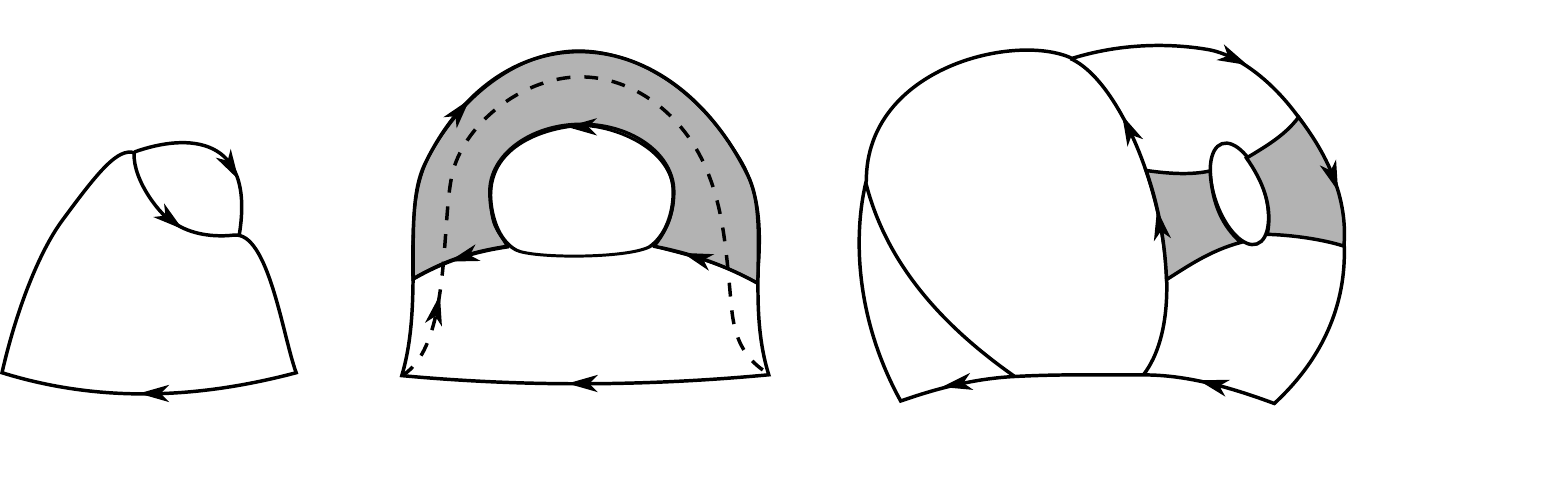\\
 \caption{}\label{figF}
\end{figure}

\begin{lem}\label{tower} The section $q_h$ is repelling in the tower $\Delta_h$ of height $h>0$.
\end{lem}
\proof Proving by contradiction we assume that $\Delta _h$ is a counterexample of minimal possible height $h\ge 1$. Suppose there is an $\cal R$-cell $\Pi$ and a subdiagram $\Delta'$
with contour $s_1t_1s_2t_2$, where $t_1$ is a section of $\partial\Pi$ of length $\ge (1-\alpha)|\partial\Pi|$,
$t_2$ is a subpath of $q_h$, $|s_1|, |s_2| <\zeta|\partial\Pi|$, and $\Delta'$ does not contain $\Pi$ (Fig. \ref{figF} b)).

Note that $\Pi=\Pi_i,$ that is $\Pi$ belongs to $\Delta_i$ but does not belong to $\Delta_{i-1}$
for some $i\ge 1.$
The paths $s_1$ and $s_2$ connecting $\partial\Pi$ with $q_h$ must have
common vertices with  $q_i$, and therefore there are subpaths $p_j$ of  $s_j$ ($j=1,2$) and  a subdiagram $\Gamma'$ in $\Delta_i$
with a contour $p_1t_1p_2t$, where $\Gamma'$ does not contain $\Pi$, $|p_j|\le |s_j|<\zeta|\partial\Pi|$ and
$t$ is a subpath of $q_i$. In other words, $\Pi$ is $(1-\alpha)$-close to $q_i$, and the minimality
of $h$ implies that $i=h$.

The same argument shows that if a cell $\Pi_j$, with $j<h$, is $(1-\alpha)$-close to the subpath $t_2$ of $q_h$
in $\Delta_h$, then it is $(1-\alpha)$-close to $q_{h-1}$ contrary to the choice of $h$. It follows that
the $\cal R$-cells of $\Delta'$ are not $(1-\alpha)$-close to $t_2$, and therefore $t_2$ is a repelling
section of $\partial\Delta'$.

The path $t_1$ is smooth in $\Delta'$ by  Lemma \ref{151}. If $\Delta'$ has no $\beta$-cells,
then we may apply Lemma \ref{nobeta} (a) to it because $2\zeta<(1-\alpha)\gamma$. We obtain a subpath $q$
of the section $t_1$ of $\partial\Pi$ with $|q|>(1-\alpha - 2\zeta\gamma^{-1})|\partial\Pi| >(1-2\alpha)|\partial\Pi|$
which is immediately close to $q_h$
This contradicts the definition of tower because the degree
of immediate contiguity of $\Pi$ to $q_{h-1}$ should be at least $2\alpha$.

Thus $\Delta'$ has a $\beta$-cell $\pi$, i.e., there are two disjoint contiguity subdiagrams $\Gamma_1$
and $\Gamma_2$ of $\pi$ to $t_1$ and $t_2$ with the sum of degrees $>\bar\beta$ and with the
contours $x^k_1y^k_1x^k_2y^k_2$  ($k=1,2$), respectively (Fig. \ref{figF} c)). The subpaths $y_2^1$ and $y_2^2$ of $t_1$ and $t_2$
both belong to $q_{h-1}$ since all the $\cal R$-cells of $\Delta'$ belong to $\Delta_{h-1}$
but $\Pi$ does not belong to it. It follows from Remark \ref{close} that $\pi$ is $\bar\beta$-close to $q_{h-1}$,
because the subdiagrams $\Gamma_1$ and $\Gamma_2$ can be included in a single subdiagram of "closeness" to $q_{h-1}$. This contradicts to the minimality of $h$ since $\bar\beta>1-\alpha$.

The lemma is proved by contradiction.
\endproof

\begin{lem} \label{bound} For a tower $\Delta_h$ of height $h$ with a base $p$, we have $|q_h|\le |p|^2$.
\end{lem}
\proof The path $q_h$ is repelling by Lemma \ref{tower}, and the statement follows from Lemma \ref{dlina}.
\endproof

\begin{rem} The stronger statement can be proved for towers in the same way if the exponent $n$ is large enough:
 $|q_h|\le |p|^{1+c}$ with a positive $c=c(n)\to 0$.
\end{rem}

\begin{rem} \label{red} Assume that the boundary label of a subdiagram $\Gamma$ of rank $0$ is $aa^{-1}aa^{-1}$ for a letter $a$,
and the boundary edges of $\Gamma$ corresponding to the first occurrences of $a$ and $a^{-1}$ are connected by a $0$-bond and so are the edges corresponding to the second occurrences of these letters (Fig. \ref{figG}). Then one can
make a diamond move, i.e. to replace $\Gamma$ by a diagram $\Gamma'$ of rank $0$ with the same label, but in $\Gamma'$,
a $0$-bond connects the boundary edges corresponding to the first $a$ and the last $a^{-1}$ and another $0$-bond corresponds to the  middle $a^{-1}a$. Obviously diamond moves preserve the $\cal R$-cells of the whole diagram,
and transform a reduced diagram to a reduced one.

\begin{figure}
 \centering\hspace*{10mm}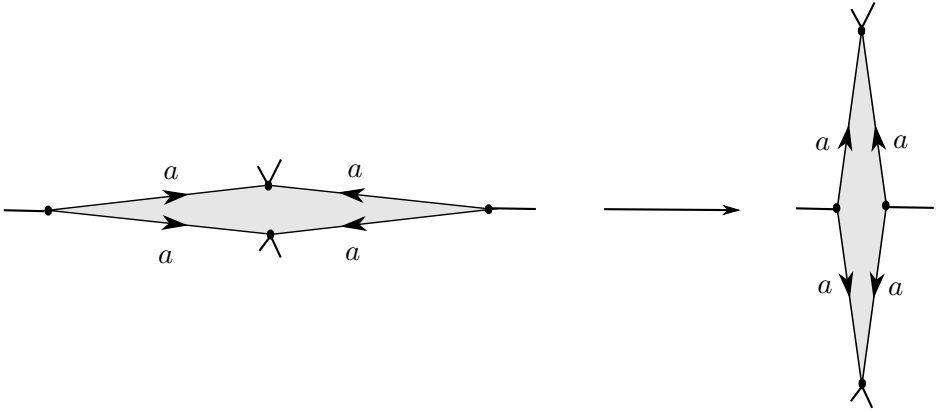\\
 \caption{Diamond move}\label{figG}
\end{figure}

We will use diamond moves as follows. Assume that $p$ is a boundary section of a reduced diagram $\Delta$,
its label is a reduced word, and a tower $\Delta_{h-1}$ with base $p$ and height $h-1$ is constructed as  a subdiagram of $\Delta$.
Let we have a cell $\Pi_h$, as in the definition of the tower, but the word $\Lab(q_h)$ is not reduced.
Then we can make a number of diamond moves making the label $q_h$ reduced in the modified diagram. These
moves increase the degree of immediate contiguity of $\Pi_h$ to $q_{h-1}.$
Therefore we will assume further, that if $\Delta_h$ is a {\em maximal (sub)tower} in $\Delta$ with base
$p$, then  the label of $q_h$ is reduced and there are no $\cal R$-cells $\Pi$ in $\Delta\backslash\Delta_h$ with the degree of immediate contiguity to $q_h$ at least $2\alpha$.
\end{rem}

\subsection{Algebraic corollaries and proof of Theorem \ref{main1}}

In this subsection we derive some algebraic results about free Burnside groups and prove Theorem \ref{main1}. Recall that $B(m,n)$ denotes the free Burnside group with basis $\cal A$ of arbitrary cardinality $m\ge 2$ and of large enough odd exponent $n$. As in the previous subsection, all van Kampen diagrams considered here are over the presentation (\ref{B}).

We will need an auxiliary infinite set of positive words $\cal W$ in the alphabet $\cal A$ of
arbitrary cardinality $m\ge 2$. It must satisfy the following conditions.

\begin{enumerate}
\item[(*)] If $W\in \cal W$ and $V$ is a subword of $W$ of length $\ge |W|/10$, then $V$
is not a subword of another word from $\cal W$ and $V$ occurs in $W$ as a subword only once.

\item[(**)] Every word $W\in \cal W$ is $11$-aperiodic: we assume that no non-empty word of the form $V^{11}$
is a subword of $W$.

\item[(***)] For every constant $C>0$, the subset of all words of length $>C$ from $\cal W$
has cardinality $\max(\aleph_0, m)$.
\end{enumerate}

Such a set ${\cal W}_2$ was constructed by D. Sonkin in \cite{S} for $m=2$. One may assume
that all the words in ${\cal W}_2$ are long enough, e.g. have length at least $1000$.  If $m>2$, then $\cal W$ is the union of the copies of ${\cal W}_2$
in all the 2-letter subalphabets of $\cal A$.

Recall that there is a standard way to define the free product $G$  of groups $G_1, G_2,\dots$ in a group variety $\cal V$. (Here the set of subscripts is not necessarily finite or countable.) The group $G\in \cal V$ is generated by the subgroups (isomorphic to) $G_i$-s, and arbitrary homomorphisms of these groups to a group $H\in \cal V$ must extend to a homomorphism $G\to H$. The group $G$ is the quotient of the usual free product $\star_i G_i$
(taken in the class of all groups) over the verbal subgroup $V(G)$ corresponding to
the laws defining the variety $\cal V$.

\begin{thm}\label{conj}
Let $n\in \mathbb N$ be a large enough odd number and let $m\ge 2$ be a cardinal number. Assume that for some positive integer $r$, we have $r$ arbitrary families $(g_{11},g_{21},\dots),\dots,$ $(g_{1r},g_{2r},\dots)$ of elements of equal cardinalities $\le \max(\aleph_0, m)$ in $B(m,n)$ (repetitions are allowed). Then
there exist elements $f_1,f_2,\dots$ of $B(m,n)$ such that each of the the subgroups $H_k$ ($k=1,\dots,r$)
generated by  the conjugates $h_{1k} =f_1g_{1k}f_1^{-1}, h_{2k}=f_2g_{2k}f_2^{-1},\dots,$
is canonically isomorphic to the free product of its cyclic subgroups
$<h_{1k}> $, $<h_{2k}>, \dots$
in the variety of groups satisfying the law $x^n=1$.
\end{thm}

\proof We may assume that all the elements $g_{jk}$ are non-trivial in $B(m,n)$.
Let they be presented by some reduced words $U_{jk}$ in the generators $a_1, a_2\dots$.
By the part 1) of \cite[Theorem 19.4]{book}, any word $U_{jk}$ is  conjugate to a power of some period (of some rank) $A_{jk}$. Replacing each $A_{jk}$ by a cyclic permutation $A'_{jk}$
we may assume that
$U_{jk}\equiv V_{jk}(A'_{jk})^{d_{jk}}V_{jk}^{-1}$, the word $U_{jk}$ is reduced, and  $d_{jk} \mid n,$ whence $n=d_{jk}n_{jk}$, where $n_{jk}$ is the order of $g_{jk}$
since by \cite[Theorem 19.4 2)]{book}, every period $A_{jk}$ has order $n$ in $B(m,n)$.
We choose the elements $f_1,f_2,\dots$ to be presented by pairwise different words
$W_1,W_2,\dots$ from $\cal W$ such that
\begin{equation}\label{Wj}
|W_j|\ge 100 n^2\max_{1\le k\le r} |U_{jk}|^2
\end{equation}
 for every
$j=1,2,\dots$.
Such a choice is possible by property (***). From now we may omit the index $k$
in $g_{jk}, h_{jk}, H_k, U_{jk}, A_{jk}, A'_{jk}, V_{jk}, d_{jk}, n_{jk},$
and so on because after the conjugators $f_1,f_2,\dots$ are chosen independently of $k$,  the statement
of the theorem can be proved separately for every subgroup $H_1,\dots,H_r.$

Thus, from now we consider only one subgroup $H$.  A word $R(h_1,h_2,\dots)$ in
the generators of $H$ is equal to $R(W_1U_1W_1^{-1}, W_2U_2W_2^{-1},\dots)$, i.e.,
being rewritten over the alphabet $\cal A$ it has the form

\begin{equation}\label{P}
P\equiv (W_{i_1}V_{i_1}(A'_{i_1})^{m_1}V_{i_1}^{-1}W_{i_1}^{-1})(W_{i_2}V_{i_2}(A'_{i_2})^{m_2}V_{i_2}^{-1}W_{i_2}^{-1})\dots (W_{i_l}V_{i_l}(A'_{i_l})^{m_l}V_{i_l}^{-1}W_{i_l}^{-1})
\end{equation}
We define the {\it type} of the product (\ref{P}) as follows. Let $\sigma_t$ be the sum of the absolute values of the exponents $m_j$
over all the occurrences of the cyclic shifts $A'_{i_j}$ of periods $A_{i_j}$ of rank $t$ in the parentheses of the right-hand side of (\ref{P}),
$\tau_t=\tau_t(P)=\sigma_t/n$ and $\tau(P)=(\tau_1,\tau_2,\dots)$.

If $P$ is trivial in $B(m,n)$, then there is a diagram $\Delta$ with the contour $q$ labeled by $P$.
The type $\tau(\Delta)=(\tau_1,\tau_2,\dots,)$, where $\tau_i$ is the number of cells of rank $i$ in $\Delta$.
We define $\tau(P,\Delta)=\tau(P)+\tau(\Delta)$, where the sum is componentwise.

Our goal is to prove that arbitrary relation $R(h_1,h_2,...)=1$ follows
in $H$ from the relations of the form $h_j^{n_j}\equiv (W_jV_j(A'_j)^{d_j}V_j^{-1}W_j^{-1})^{n_j}=1$
and the relations of the form $v(h_1,\dots, h_s)^n=1$, where $v$ is any word.
Proving by contradiction, we choose the word $R(h_1,h_2,...)$, its form $P$ (\ref{P})
and a diagram $\Delta$ with boundary label $P$ such that the type $\tau(P,\Delta)$
is as low as possible. In particular, given boundary label $P$, the type of $\Delta$ is minimal, and so $\Delta$ is a reduced diagram.

If $W_{i_s}\equiv W_{i_{s+1}}$, then $A'_{i_s}\equiv A'_{i_{s+1}}$, $V_{i_s}\equiv V_{i_{s+1}}$, and the
(sub)product  \\ $$(W_{i_s}V_{i_s}(A'_{i_s})^{m_s}V_{i_s}^{-1}W_{i_s}^{-1})(W_{i_{s+1}}V_{i_{s+1}}(A'_{i_{s+1}})^{m_{s+1}}V_{i_{s+1}}^{-1}W_{i_{s+1}})$$
is freely equal to $W_{i_s}V_{i_s}(A'_{i_s})^{m_s+m_{s+1}}V_{i_{s+1}}^{-1}W_{i_s}^{-1}$. The new product $P'$
has type $\tau(P')\le\tau(P)$, and so we will assume that $W_{i_s}\ne W_{i_{s+1}}$
for every $s$. Similarly we may assume that $0<|m_s|<n$. Indeed, if $m_s\ge n$ (if $m_s\le -n$),
then one can add one cell of rank $t=r(A_{i_s})$ to $\Delta$ (and replace the obtained diagram by a
minimal one) and replace the
occurrence $(A'_{i_s})^{m_s}$ in $P$ by $(A'_{i_s})^{m_s-n}$ (by $(A'_{i_s})^{m_s+n})$, resp.). This transformation
corresponds to inserting of $h_{i_s}^{\pm n_{i_s}}$ in the word $R$, and it
decreases $\tau_t(P)$ by $1$ and increases $\tau_t(\Delta)$ by at most $1$, so $\tau(P,\Delta)$
does not increase.

\begin{figure}
 \centering\hspace*{10mm}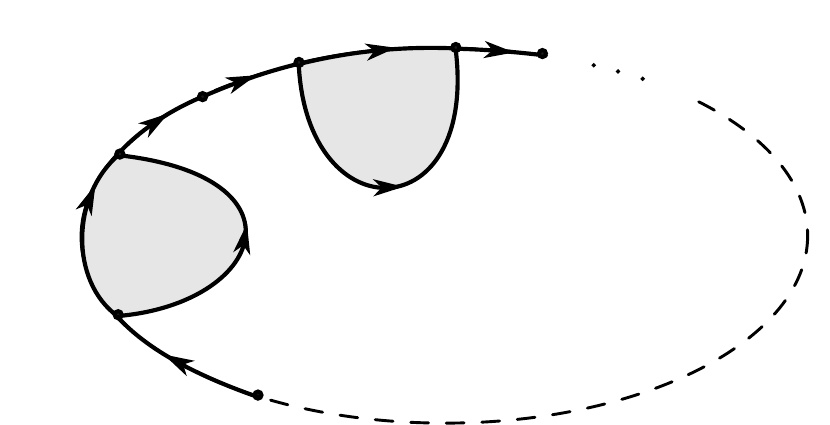\\
 \caption{}\label{figH}
\end{figure}

We have the boundary section $q = w_1 p_1 w'_1\dots w_lp_lw'_l$ in $\Delta$, where $\Lab(w_s)\equiv W_{i_s}$, $\Lab(p_s)\equiv V_{i_s}(A'_{i_s})^{m_s}V_{i_s}^{-1}$, and
$\Lab(w'_s) \equiv W_{i_s}^{-1}$ ($1\le s \le l$). Denote by $\Delta_1$ a maximal subdiagram of $\Delta$
which is a tower with the base $p_1$. Let $p_1q_1^{-1}$ be the contour of $\Delta_1$. Then we denote by $\Delta_2$
a maximal tower with the base $p_2$ in the subdiagram with the contour $w_1 q_1 w'_1\dots w_lp_lw'_l$,
it is bounded by $p_2q_2^{-1}$, and so on (see Fig. \ref{figH}). We obtain a reduced diagram $\Delta'$ with contour
$q'=w_1 q_1 w'_1\dots w_lq_lw'_l$, where the words $Q_s\equiv \Lab(q_s)$ are reduced and $\Delta'$ has no cells $\Pi$
with the degree of immediate contiguity to $q_j $ at least $2\alpha$ by Remark \ref{red}.
Also note that every $Q_s$ is nontrivial since $\Lab(p_s)=\Lab(q_s)$ in $B(m, n)$.
By Lemma \ref{dlina}, $|q_s|\le |p_s|^2<|V_{i_s}(A'_{i_s})^{m_s}V_{i_s}^{-1}|^2<(n|V_{i_s}A'_{i_s}V_{i_s}^{-1}|)^2$, and so by the choice (\ref{Wj}) of the words $W_1,W_2,\dots$, we have

\begin{equation}\label{vyrez}
|w_s|=|W_{i_s}|>100 |q_s|,\;\;\; s=1,\dots, l
\end{equation}

The possible cancelations in the word $W_{i_s}Q_sW_{i_s}^{-1}$ can affect a suffix of
length $<11|Q_s|$ in  $W_{i_s}$ since by (**), the word $W_{i_s}$ does not contain non-trivial
$11$-th powers.  The possible cancelations in the products $W_{i_s}^{-1}W_{i_{s+1}}$ can
touch less then $1/10$ of each  factors by (*). Thus after all the cancelations
in $\Lab(q)$ (they correspond to diamond moves in $\Delta'$), we will have a reduced
diagram $\bar\Delta$ with a reduced boundary label $\bar Q$ of $\bar q = \bar w_1 \bar q_1 \bar w'_1\dots \bar w_l\bar q_l\bar w'_l,$ where $\bar W_s\equiv \Lab (\bar w_s)$ and $\bar W'_s\equiv \Lab (\bar w'_s)^{-1}$ are subwords of $W_{i_s}^{\pm 1}$ with
length $>\frac34 |W_{i_s}|$, $|\bar q_s|< \frac1{80}\min (|\bar w_s|,|\bar w'_s|)$ for $s=1,\dots, l$,
and there are no cells immediately close in $\bar\Delta$ to $\bar q_s$ with degree $\ge 2\alpha$. In other words, The boundary label $\bar Q$ consists of long {\it traces}
$\bar W_s$ and $\bar W'_s$ of the conjugating words $W_{i_s}^{\pm 1}$ and of
short intervals $\bar Q_s$ between them ($s=1,\dots,l$).

Since the word $\bar Q$ is reduced and non-empty, we have $r(\bar\Delta)>0$, and by Lemma \ref{nobeta} (b),
$\bar\Delta$ has a cell $\Pi$ immediately close to $\bar q$ with degree $>2/5$.
We will use the cell $\Pi$ to decrease the type $\tau(P,\Delta)$ of the original counter-example.

Let $\Lab(\partial\Pi)$ have a period $A$, that is, the arc $t$ of $\partial\Pi$ immediately close to $\bar q$
is labeled by an $A$-periodic word of length $>\frac{2n}5|A|$. Any subarc of
length $>2\alpha n |A|$ of $t$ cannot be immediately close to some $\bar q_s$ since $\Delta_s$ is the maximal subtower of $\Delta$ with the base $p_s$. A subarc of $t$ whose label is a subword of some word $\Lab(\bar w_s\bar w'_{s+1})$ must be of length $<22|A|$ by the condition (**). Hence $t^{-1}$ has a subpath $z$ of length
$>(2/5-4\alpha)n|A|> \frac{n}{3}|A|$ whose label starts and ends with a whole word of the form $\Lab(\bar w_s)$ or $\Lab(\bar w'_{s+1})$.

Let $W_{i_k}^{\pm 1}$ be the longest among the words whose traces $\bar W_s,$ $\bar W'_s$ occur in the word $Z\equiv \Lab(z),$ and $W$ is its trace.  Without lost of generality we assume that $W\equiv \bar W_k$ for some $k$. so we have $Z\equiv Z_1WZ_2.$
But since the word $Z$ is $A$-periodic and $|W|<11|A|$ one can shift this occurrence of $W$ to the right (or to the
left): $Z\equiv Z_3WZ_4$ with $|Z_3|=|Z_1|+|A|.$ By the choice of $k$ and the inequality
(\ref{vyrez}), we have $|W|>\frac34 |W_{i_k}|> 75|\bar Q_s|$ for the labels $\bar Q_s$ of arbitrary subpath $\bar q_s$ occuring in $z$.
Hence the occurrence of $W$ in $Z_3WZ_4$ has to overlap  with a  trace $\bar W$ of some $W_{i_r}^{\pm 1}$ by a subword $V$ of length $|V|>\frac13 \min (|W|, |\bar W|),$
Since $\frac13\cdot \frac34 >\frac {1}{10}$, we have $W_{i_r}\equiv W_{i_k}$
by the property (*). It also follows that $r>k$.

 Thus, both occurrences of $W$ in $Z$ are the traces of the same word $W_{i_k}$  but with different occurrences in the product $P$. Therefore the  period $A,$
 is freely conjugate to a word $\bar A\equiv W_{i_k}\bar Q_k \bar W'_k \bar W_{k+1}\bar Q_{k+1} \bar W'_{k+1}\dots \bar W_{u}\bar Q_u W_{i_u}^{-1}$ for some $u>k$.

 Hence the word $\bar A$ is equal in $B(m,n)$ to the subword $$U\equiv W_{i_k}V_{i_k}(A'_{i_k})^{m_k}V_{i_k}^{-1}W_{i_k}^{-1}\dots W_{i_u}V_{i_u}(A'_{i_u})^{m_u}V_{i_u}^{-1}W_{i_u}^{-1}$$
of the product $P\equiv U_1UU_2$, and each of the factors $U_1$, $U$ and $U_2$ represent
an element from the subgroup $H$. Moreover $\bar A= U$ in  rank $i$, where $i$
is the maximum of the ranks of the towers $\Delta_j$ with the bases $p_j$, where $k\le j\le u $. Recall that by the choice of $k$ and (\ref{vyrez}), $|q_j|\le |W_{i_k}|/100$. Hence $|\partial\Delta_j|< |W_{i_k}|/50$, and so
$$i=max_{k\le j\le u} r(\Delta_j)< (\bar\beta n)^{-1} |W_{i_k}|/50< (4n)^{-1}|W_{i_k}|$$ by Lemma \ref{beta}.
On the other hand, $r(\Pi)=|A| > |W| > \frac34 |W_{i_k}|$, i.e., $r(\Pi)>i$.

We obtain that the boundary label of $\Pi$ is freely conjugate to $\bar A^n$ which is equal to $U^n$ in rank $i$, and  if we remove $\Pi$
together with $\Delta_k,\dots,\Delta_u$ from $\Delta$, we obtain a diagram $\Delta^1$, with boundary label freely equal to the product $P(1)$ obtained from the product $P$ by the replacement of $U$ by $\bar A^{1-n}$. Attaching $n-1$ copies of diagrams for the equality $\bar A = U$ in rank $i$, we obtain a diagram $\Delta^2$ with boundary label freely equal to $P(2)$ obtained from $P$ by the substitution $U\to U^{1-n}$. Note that the words $U$ and  $P(2)$ are words in the generators of $H$. Finally, we replace $\Delta^2$ by a diagram  $\Delta^3$ of minimal type with the same label $P(3)\equiv P(2).$
Note that $\tau(\Delta^3)<\tau(\Delta)$ since we removed a cell of rank $|A|$ and added a number of cells
of rank $\le i<|A|$. Moreover $\tau_{|A|}(\Delta^3)<\tau_{|A|}(\Delta)$. On the other hand, the transition
$U\to U^{1-n}$ can change the type of $P$ only for the components $\tau_j(P)$ with $j<|A|$. Indeed,
for $k\le j \le u$, we have by (\ref{vyrez}): $$r (A_{i_j})= |A_{i_j}|<|W_{i_j}|/100 <\frac43\frac{|\bar W_{j}|}{100}\le |A|/75=r(A)/75<r(A).$$

Therefore $\tau(P(3),\Delta^3)<\tau(P,\Delta)$, but the boundary label of $\Delta^3$ is $U_1U^{1-n}U_2$, which is equivalent to the word $P\equiv U_1UU_2$ modulo the relation $U^n=1$, where the words $U, U_1, U_2$ represent elements from the subgroup $H$. This contradicts to the minimality
of the type $\tau(P,\Delta)$ in our counter-example. The theorem is proved.
\endproof

\begin{cor}\label{Burns-cor}
\begin{enumerate}
\item[(a)] If under the hypothesis of Theorem \ref{conj}, all the elements
$g_{1k},g_{2k},\dots$ ($k\in\{1,\dots,r\}$) have order $n$, then the conjugate elements $h_{1k},h_{k2},...$ form
a basis in the free Burnside subgroup $H_k$. In particular, this is always the case if $n$ is prime.

\item[(b)] For every finite $m\ge 2$ and every big enough odd $n$, there is a free Burnside subgroup $H\le B(m,n)$ of infinite rank having non-empty intersection with every conjugacy class of $B(m,n)$.
\end{enumerate}
\end{cor}

\proof The first claim directly follows from Theorem \ref{conj}. For (b), we just note
that every non-trivial element of $B(m,n)$ belongs to a cyclic subgroup of order $n$
(e.g., \cite[Theorem 19.4]{book}), and so it suffices to apply Theorem \ref{conj} to
the family of elements of order $n$.\endproof

\begin{thm}\label{commutators}
Let $n\in \mathbb N$ be a large enough odd number and let $m\ge 2$ be a cardinal number. Assume that for some positive integer $r$, we have $r$ arbitrary families $(g_{11},g_{21},\dots)$, $\dots$, $(g_{1r},g_{2r},\dots)$ of nontrivial elements of $B(m,n)$ with equal cardinalities $\le \max(\aleph_0, m)$. Then
there exist elements $f_1,f_2,\dots$ of $B(m,n)$ such that for every $k\in\{1,\dots,r\}$ the commutators
$h_{1k} =f_1g_{1k}f_1^{-1}g_{1k}^{-1}, h_{2k}=f_2g_{2k}f_2^{-1}g_{2k}^{-1},\dots $
freely generate a free Burnside
subgroup of exponent $n$.
\end{thm}

\proof This is an analog of Theorem \ref{conj}, and we choose the words $U_j$ and $W_j$ as there.
Let $P$ be a product of the commutators of the form $(W_{i_j}U_{i_j}W_{i_j}^{-1}U_{i_j}^{-1})^{\pm 1}$. Our goal now is to prove that an equality $P=1$ in $B(m,n)$ follows from the relations of the subgroup $H=<h_1,h_2,\dots>$ of the form $v(h_1,h_2,\dots)^n=1.$ Again we consider the diagram $\Delta$ over the presentation of $B(m,n)$ with
minimal type that corresponds to a nontrivial equality $P=1$. (We do not introduce $\tau(P)$ now.) Its contour
$q$ has a factorization $q = \prod (w_jp_jw'_jp'_j)^{\pm 1},$ where the sections $w_j, w'_j$ have labels $W_{i_j}^{\pm 1}$, $\Lab(p_j)\equiv U_{i_j}$, and $\Lab(p'_j)$ is either  $U_{i_j}^{-1}$, or the reduced form of   $U_{i_j}^{-1}U_{i_{j+1}}$, or of $U_{i_{j-1}}^{-1}U_{i_{j}}.$ All these labels are nonempty
by the choice of the words $W_1,W_2,\dots$

Then, as in Theorem \ref{conj}, we construct the towers
based now on the subpaths $p_j$ and $p'_j$, and removing them we obtain a diagram $\bar\Delta$.
As before, we obtain a cell $\Pi$ with a boundary  arc $t$
which is immediately close to $\bar q,$ and $\Lab(t^{-1})$
starts and ends with some $\Lab(\bar w_s)^{\pm 1}$ or with  $\Lab(\bar w'_s)^{\pm 1}$. Recall also that
arbitrary subpath $\bar q_s$ has length $<2\alpha n|A|$ and $|\bar w_s|, |\bar w'_s|<11|A|.$ It follows that the arc
$z^{-1}$ of length $>(2n/5-44-12\alpha n)|A|>n|A|/3$ can be chosen starting and ending with different subpaths of the form $(\bar w_sp_s\bar w'_s)^{\pm 1}$. By the small cancellation argument (i.e.,we use (*) as earlier), such a word uniquely determines
a commutator $[W_{i_k},U_{i_k}]^{\pm 1}$ because $U_j\ne U_j^{-1}$ in $B(m,n)$ for the nonidentity element $U_j$ of this group having odd exponent. Now we consider the shift of the occurrence $\Lab(\bar w_sp_s\bar w'_s)^{\pm 1}$
in $\Lab(z^{-1})$ by $|A|$ and finish the proof as in Theorem \ref{conj}.
\endproof

To prove Theorem \ref{main1} we need the following particular case. We use the notation $[x,y]=xyx^{-1}y^{-1}$.

\begin{cor}\label{cor-com}
Let $B(m,n)$ be the free Burnside group of large enough odd exponent $n$ and at most countable rank $m\ge 2$. Then there exists a family $\mathcal Y=\{Y_i\}_{i\in \mathbb N}$ of finite subsets $Y_i\subset B(m,n)$ such that $|Y_i|\to \infty $ as $i\to \infty$ and for every non-trivial element $g\in B(m,n)$, the set of commutators $\{[y,g]\mid y\in Y_i\}$ is a basis of a free Burnside subgroup of $B(m,n)$ of exponent $n$ and rank $|Y_i|$ for all but finitely many $i$.
\end{cor}

\begin{proof}
Let $B(m,n)=\{1=b_0, b_1, b_2, \ldots\}$ and let $i\in \mathbb N$. By Theorem \ref{commutators} applied to the constant sequences $g_{jk}=b_k$, $k=1, \ldots , i$, there exist elements $f_1, \ldots, f_i$ such that for every $k=1, \ldots , i$, the commutators $[f_1,b_k], \ldots, [f_i, b_k]$ are pairwise distinct and form a basis of a free Burnside subgroup of $B(m,n)$. Let $Y_i=\{ f_1, \ldots , f_i\}$. The collection $\mathcal Y=\{ Y_i\}_{i\in \mathbb N}$ obviously satisfies the required property.
\end{proof}

\begin{rem} Note that free Burnside groups of exponent $n$ may contain subgroups which are not free in the variety ${\cal B}_n$ of all groups with the law $x^n=1$ . This follows from a theorem of P. Neumann and J. Wiegold (see \cite{N}, Theorem
43.6) for any exponent $n>2$. Moreover no nontrivial normal subgroup of $B(m, n)$ is free in the variety ${\cal B}_n$ if the exponent $n$ is a large enough odd integer
(see \cite{I} for composite exponents and \cite{Ols03} for prime ones).
\end{rem}

\begin{proof}[Proof of Theorem \ref{main1}]
We first assume that $m$ is at most countable. Let $\mathcal Y$ be the family of subsets chosen in accordance to Corollary \ref{cor-com}. Fix any $g\in G\setminus\{ 1\}$. Since $\lambda_G(g)$ is unitary,  we have
\begin{equation}\label{Be1}
\left\|\sum\limits_{y\in Y_i}\lambda_G(ygy^{-1})\right\| = \left\| \left(\sum\limits_{y\in Y_i}\lambda_{G}(ygy^{-1})\right)\lambda_G(g^{-1})\right\|= \left\|\sum\limits_{y\in Y_i}\lambda_{G}([y,g])\right\|.
\end{equation}
Let $H_{g,i}$ denote the subgroup of $G$ generated by the set $T_{g,i}=\{ [y,g]\mid y\in Y_i\}$. By part (a) of Lemma \ref{norms}, we have
\begin{equation}\label{Be2}
\left\|\sum\limits_{y\in Y_i}\lambda_{G}([y,g])\right\| = \left\|\sum\limits_{y\in Y_i}\lambda_{H_{g,i}}([y,g])\right\|.
\end{equation}
By Corollary \ref{cor-com}, $H_{g,i}$ is free Burnside of exponent $n$ with basis $T_{g,i}$ for all but finitely many $i$. Hence the sequence $\{ (H_{g,i}, T_{g,i})\}_{i\in \mathbb N}$ has infinitesimal spectral radius by Corollary \ref{bisr}. Combining this with (\ref{Be1}) and (\ref{Be2}) yields
$$
\lim_{i\to\infty}\frac1{|Y_i|}\left\|\sum\limits_{y\in Y_i}\lambda_G(ygy^{-1})\right\| =0.
$$
Thus Lemma \ref{AL} applies.

For an uncountable group $G$,  $C^\ast_{red}(G)$ is the union of $C^\ast_{red}(H)$ over all countable subgroups $H$ of $G$. If $G$ is generated by a set $X$, then every countable subgroup $H\le G$ belongs to a subgroup $K_H$ generated by a countable subset of $X$. Replacing $H$ with $K_H$ if necessary and applying this to $G=B(m,n)$, we obtain that $C^\ast_{red} (G)$ is the union of subalgebras isomorphic to $C^\ast_{red} (B(m,n))$ for at most countable cardinals $m$. Thus the general case of the theorem follows from the countable one.
\end{proof}


\section{$C^\ast$-simple limits of relatively hyperbolic groups}


\subsection{Relatively hyperbolic groups}
Let $G$ be a group generated by a subset $X$. In this section we denote by $|g|_X$ the word length of an element $g\in G$ with respect to $X$.

We recall one of many equivalent definitions of relatively hyperbolic groups; for a discussion of other definitions and their relationship see \cite{Hru,Osi06}. Let $G$ be a group, $\Hl $ a collection subgroups of $G$. Let also $X$ be a subset of $G$ such that $G$ is generated by $X$ together with
the union of all $H_\lambda $; such a subset $X$ is called a \emph{relative generating set} of $G$ with respect to $\Hl$. Then the group $G$ is naturally a quotient of the free product
\begin{equation}
F=\left( \ast _{\lambda\in \Lambda } H_\lambda  \right) \ast F(X),
\label{F}
\end{equation}
where $F(X)$ is the free group with the basis $X$. Let
\begin{equation}\label{H}
\mathcal H=\bigsqcup\limits_{\lambda\in \Lambda} (H_\lambda
\setminus \{ 1\} ) .
\end{equation}
Here we think of $H_\lambda$ as subgroups of $F$, so the union in (\ref{H}) is indeed disjoint. By abuse of notation, we will identify $\mathcal H$ and $H_\lambda$ with their images under the natural homomorphism $F \to G$. Note that the restriction of this map to $\mathcal H$ is not necessarily injective.

Suppose that the kernel of
the natural homomorphism $F\to G$ is the normal closure of a subset
$\mathcal R$ in the group $F$. In this case we say that $G$ has {\it relative
presentation}
\begin{equation}\label{G}
\langle X,\; H_\lambda, \lambda\in \Lambda \; | \; \mathcal R
\rangle .
\end{equation}
If $|X|<\infty $ and $|\mathcal R|<\infty $, the
relative presentation (\ref{G}) is said to be {\it finite}.
Further for every word $W$ in the alphabet $X^{\pm 1} \cup \mathcal H$ such that
$W=_G1$ in $G$, there exists an expression
\begin{equation}
W{=_F} \prod\limits_{i=1}^k f_i^{-1}R_i^{\e_i}f_i \label{prod}
\end{equation}
with the equality in the group $F$, where $R_i\in \mathcal R$, $\e_i=\pm 1$, and
$f_i\in F $ for $i=1, \ldots , k$. The
{\it relative area} of $W$, denoted $Area^{rel}(W)$, is the smallest possible number
$k$ in a representation of the form (\ref{prod}).

A group $G$ is {\it hyperbolic relative to a collection of
subgroups} $\Hl $, called {\it peripheral subgroups} (or \emph{peripheral structure}) of $G$, if there exists a finite relative presentation (\ref{G}) and a constant $C$ such that for any $n\in \mathbb N$ and any word $W$ of length at most $n$ in the alphabet $X^{\pm 1}\cup
\mathcal H$ representing the identity in $G$, we have
$Area^{rel} (W)\le Cn $. In particular, $G$ is an ordinary {\it hyperbolic group} if $G$ is hyperbolic relative to the empty family of peripheral subgroups (or relative to the trivial subgroup).

We will need several basic facts about relatively hyperbolic groups.

\begin{lem}[\cite{Osi06}, Theorem 1.4]\label{maln}
Let $G$ be a group hyperbolic relative to a collection of subgroups $\Hl $. Then the following conditions hold.
\begin{enumerate}
\item[(a)] For every distinct $\lambda, \mu \in \Lambda $ and every $g\in G$, we have $|H_\lambda \cap H_\mu^g |<\infty $.

\item[(b)] For every $\lambda \in \Lambda $ and every $g\in G\setminus H_\lambda $, we have $|H_\lambda \cap H_\lambda ^g|<\infty $.
\end{enumerate}
\end{lem}

The first claim of the next result is a simplification of \cite[Corollary 1.14]{DS}. The second claim follows immediately from the first one as every hyperbolic group is hyperbolic relative to the empty set of subgroups; it is also a particular case of \cite[Theorem 2.40]{Osi06}, which is proved for all (not necessary finitely generated) relatively hyperbolic groups. In fact, the first claim can also be proved for all relatively hyperbolic groups by using the methods of \cite{Osi06}. However we do not need this in our paper, so we leave this generalization to the reader.

\begin{lem}\label{240}
Let $G$ be a finitely generated group hyperbolic relative to a collection of subgroups $\Hl\cup \{H\}$. Suppose that $H$ is hyperbolic relative to a collection $\{ K_\mu\}_{\mu\in M}$. Then $G$ is hyperbolic relative to $\Hl \cup \{ K_\mu\}_{\mu\in M}$. In particular, if $H$ is hyperbolic then $G$ is hyperbolic relative to $\Hl$.
\end{lem}

The next theorem is the main result of \cite{Osi06a}.

\begin{thm}\label{Q}
Let $G$ be a relatively hyperbolic group with peripheral collection
$\Hl $ and let $X$ be a relative generating set of $G$ with respect to $\Hl$. Suppose that $Q$ is a subgroup of $G$ such that the following
conditions hold.
\begin{enumerate}
\item[(Q1)] $Q$ is generated by a finite set $T$.

\item[(Q2)] There exist $\lambda, c\in \mathbb R$ such that for any
element $q\in Q$, we have $|q|_T\le \lambda |q|_{X\cup \mathcal
H}+c$.

\item[(Q3)] For any $g\in G\setminus Q$, we have
$|Q\cap Q^g|< \infty $.
\end{enumerate}
Then $G$ is hyperbolic relative to $\Hl\cup\{ Q\}$.
\end{thm}

Recall that an element $g$ of a relatively hyperbolic group $G$ is \emph{loxodromic} if it has infinite order and is not conjugate to an element of any of the peripheral subgroups of $G$. By \cite[Theorem 4.3]{Osi06a}, every loxodromic element $g\in G$ is contained in a unique maximal virtually cyclic subgroup of $G$ denoted $ E_G(g)$; moreover, we have
\begin{equation}\label{EGg}
E_G(g)=\{ x\in G\mid \exists n\in \mathbb N \; x^{-1}g^nx=g^{\pm n}\}.
\end{equation}

We will need two corollaries of Theorem \ref{Q}. The first one was proved in \cite{Osi06a} by verifying the assumptions of Theorem \ref{Q} for $Q=E_G(g)$.

\begin{cor}[{\cite[Corollary 1.7]{Osi06a}}]\label{Eg}
Suppose that a group $G$ is hyperbolic relative to a collection of
subgroups $\Hl $. Then for every loxodromic element $g\in G$,  $G$ is also hyperbolic relative to $\Hl\cup \{ E_G(g)\}$.
\end{cor}

The second corollary is new.

\begin{cor}\label{mfs}
Let $G$ be a group, $A\subseteq G\setminus \{ 1\}$ a finite subset. Let
$$
P=G\ast F(X)\ast F(Y)\ast F(Z),
$$
where $F(X)$, $F(Y)$, $F(Z)$ are free groups with bases $X=\{ x_1, \ldots , x_n\}$, $Y=\{ y_a\mid a\in A\}$, and $Z=\{ z_a\mid a\in A\}$, respectively. Then the set
\begin{equation}\label{T}
T=\{ y_ax_iax_i^{-1} z_a\mid a\in A,\, i= 1, \ldots n\}
\end{equation}
is a basis of a free subgroup $F$ of $P$ of rank $n|A|$. Moreover, $P$ is hyperbolic relative to $\{G,F\}$.
\end{cor}

\begin{proof}
Throughout this proof, by the \emph{normal form} of an element of $p\in P$ we mean the normal form of $p$ with respect to the decomposition of $P$ into the free product of $G$ and cyclic subgroups generated by elements of $X\cup Y\cup Z$. That is, the normal form of $p$ is a (possibly empty) sequence $(p_1, \ldots, p_m)$ of non-trivial elements $p_1, \ldots, p_m\in P$ such that $p=p_1\cdots p_\ell$, every $p_i$ belongs to $G$ or to a cyclic factor $\langle w\rangle $ for some $w\in X\cup Y\cup Z$, and no $p_i$, $p_{i+1}$ belong to the same factor for $i=1, \ldots, \ell-1$. Recall that the normal form of any product $pq\cdots z$ of elements $p,q, \ldots , z\in P$ can be obtained from the concatenation of the normal forms of $p,q, \ldots , z$ by the standard reduction process which involves cancellation and consolidation. For details we refer to \cite{LS}. We say that a subsequence $(x_1, \ldots , x_k)$ of the normal form of one of the multiples $p,q, \ldots , z$ \emph{survives} in $pq\cdots z$ if it remains untouched by the reduction process. In particular, $(x_1, \ldots , x_k)$ is a subsequence of the normal form of $pq\cdots z$.

Given $a\in A$ and $i\in \{1, \ldots, n\}$, let $t_{a,i}=y_ax_iax_i^{-1} z_a$. Obviously $(y_a, x_i, a, x_i^{-1}, z_a)$ is the normal form of $t_{a,i}$. We call the subsequence $c(t_{a,i})= (x_i, a, x_i^{-1})$ the \emph{core} of $t_{a,i}$. The following claim is quite obvious and can be easily proved by induction on $k$. This is straightforward and we leave it as an exercise for the reader.

{\noindent \bf Claim. } \emph{Let $f=s_1\ldots s_k$, where $s_i\in T\cup T^{-1}$, be a freely reduced word in the alphabet $T\cup T^{-1}$.}
\begin{enumerate}
\item[(a)] \emph{For every $i\in \{ 1, \ldots, k\}$, the core $c(s_i)$ survives in $f$. We call the corresponding subsequence of the normal form of $f$ the {\rm trace} of $s_i$.}

\item[(b)] \emph{Suppose that the normal form of $f$ contains a subsequence $(g_1,g_2,g_3)=c(t)$ for some $t\in T\cup T^{-1}$. Then $(g_1,g_2,g_3)$ is a trace of some $s_i=t$.}
\end{enumerate}

The first part of the claim easily implies that  $F$ is indeed free with basis $T$. Moreover, for every $p\in P$ with normal form of length $\ell\ge 1$, we have
$$
|p|_T<\ell \le |p |_{G\cup X\cup Y\cup Z}.
$$
Hence the embedding of metric spaces $(F, |\cdot |_T)\to (G, |\cdot |_{G\cup X\cup Y\cup Z})$ induced by the inclusion $F\le G$ is Lipschitz. Obviously $P$ is hyperbolic relative to $G$ by the definition.  Thus it remains to show  that $F$ is malnormal; then application of Theorem \ref{Q} completes the proof.

Assume that $u,w\in F\setminus \{ 1\}$ and let $p$ be an element of $P$ such that $pup^{-1}=w$. Let $u=r_1\ldots r_m$, where $r_1, \ldots , r_m\in T\cup T^{-1}$ and suppose that $r_1\ldots r_m$ is reduced as a word in $T\cup T^{-1}$. By the first part of the claim the cores $c(r_i)$ survive in $u$. Passing to powers of $u$ and $w$ if necessary, we can assume that the the normal form of $u$ is long enough to guarantee the existence of $1\le i\le m$ such that the core $c(r_i)$ survives in $pup^{-1}$. Let $(g_1,g_2,g_3)$ be the corresponding subsequence of the normal form of $pup^{-1}$. Since $pup^{-1}\in F$, we have $pup^{-1}=s_1\cdots s_k$ for some freely reduced word $s_1\cdots s_k$ in the alphabet $T\cup T^{-1}$. By the second part of the claim, $(g_1,g_2,g_3)$ is the trace of some $s_j$. Therefore, $pr_1\cdots r_i=s_1\cdots s_j$ (note that the core $c(t)$ uniquely defines the element $t\in T\cup T^{-1}$). This implies that $p\in F$ and hence $F$ is malnormal.
\end{proof}

\subsection{Small cancellation and Dehn filling in relatively hyperbolic groups}

In this subsection we briefly review the results from \cite{Osi10,Osi07,Osi06a} necessary for the proof of Theorem \ref{main2}. We begin with a reformulation of \cite[Lemma 4.4]{Osi06a}.

\begin{lem}\label{ah}
Let $G$ be a group hyperbolic relative to a collection of subgroups $\Hl$. Then for every $\lambda\in \Lambda$ and every $a\notin H_\lambda$, $ah$ is loxodromic for all but finitely many $h\in H_\lambda$.
\end{lem}

Note that in \cite[Lemma 4.4]{Osi06a} the element $a$ is required to satisfy $|a|_{X\cup \mathcal H}=1$. This can always be achieved by adding $a$ to $X$. Indeed if $G$ satisfies the definition of relative hyperbolicity discussed above with relative generating set $X$, then it satisfies this definition with any other relative generating set (with respect to $\Hl$) in place of $X$, see \cite[Theorem 3.24]{Osi06}.

Lemma \ref{ah} can be used to derive the following.

\begin{cor}[{\cite[Corollary 4.5]{Osi06a}}]\label{loxex}
If an infinite group is hyperbolic relative to a collection of proper subgroups, then it contains loxodromic elements. \end{cor}

Recall that two loxodromic elements $g_1, g_2$ of a group $G$ are called \emph{commensurable} if some non-trivial powers of $g_1$ and $g_2$ are conjugate in $G$. The following notion introduced in \cite{Osi10} plays an important role in the proof of Theorem \ref{main3}.

\begin{defn}A subgroup $S$ of a relatively hyperbolic group $G$ is \emph{suitable} (with respect to the given peripheral structure of $G$) if there exist two non-commensurable loxodromic elements $g_1, g_2\in S$ such that $E_G(g_1)\cap E_G(g_2)=\{ 1\}$.
\end{defn}

The next lemma provides an equivalent characterization. It is proved in \cite{AMO}, see Lemma 3.3 and Proposition 3.4 there.

\begin{lem}\label{AMO-suit}
Let $G$ be a relatively hyperbolic group. A subgroup $S\le G$ is suitable if and only if it is not virtually cyclic, contains at least one loxodromic element, and does not normalize any non-trivial finite subgroup of $G$.
\end{lem}

Recall also that a subgroup (respectively, an element) of $G$ is called \emph{parabolic } if it is conjugate to a subgroup (respectively, an element) of $H_\lambda$ for some $\lambda \in \Lambda$.

\begin{cor}\label{Ksuit}
Let $G$ be a torsion free group hyperbolic relative to a collection of subgroups $\Hl$, $K$ a non-cyclic non-parabolic subgroup of $G$. Then $K$ is suitable.
\end{cor}

\begin{proof}
As every torsion free virtually cyclic group is cyclic, $K$ is not virtually cyclic. Consider any $x\in K\setminus\{1\}$. Since $G$ is torsion free, $x$ is either loxodromic or parabolic. In the former case $K$ is suitable by Lemma \ref{AMO-suit}. In the latter case, passing from $K$ to its conjugate if necessary, we can assume that $x\in H_\lambda$ for some $\lambda\in \Lambda$. (Observe that passing to conjugate subgroups preserves suitability.) Since $K$ is not parabolic, there exists $a\in K\setminus H_\lambda$. Then by Lemma \ref{ah}, there exists $n$ such that $ax^n$ is loxodromic. As $ax^n\in K$, we obtain that $K$ is suitable by Lemma \ref{AMO-suit} again.
\end{proof}

Finally, we will need another lemma, which is a combination of Lemma 3.8 and Lemma 3.3 in \cite{AMO}; it was also proved in more general settings in \cite[Theorem 2.23]{DGO}.

\begin{lem}\label{KG}
Let $G$ be a relatively hyperbolic group, $S\le G$. Assume that $S$ is not virtually cyclic and contains loxodromic elements. Then the following hold.
\begin{enumerate}
\item[(a)] There exists a maximal normal finite subgroup of $G$ normalized by $S$, which we denote by $K(S)$.
\item[(b)] $S$ contains a loxodromic element $g$ such that $E_G(g)=\langle g\rangle \times K(S)$.
\end{enumerate}
\end{lem}

\begin{rem}\label{S=G} If $G$ is not virtually cyclic  and is hyperbolic relative to a collection of proper subgroups, then $G$ contains loxodromic elements by Corollary \ref{loxex}, and we can apply this lemma to $S=G$. In particular, we can apply Lemma \ref{KG} in the case when a relatively hyperbolic group $G$ contains a suitable subgroup as the existence of such a subgroup automatically implies that $G$ is not virtually cyclic and all peripheral subgroups are proper.
\end{rem}

We now state a theorem from \cite{Osi10}, which was proved by means of small cancellation theory in relatively hyperbolic groups. The idea of generalizing small cancellation theory to groups acting on hyperbolic spaces goes back to Gromov's paper \cite{Gro}. In the case of hyperbolic groups, it was formalized by several people including the first author \cite{Ols93}; later the second author generalized this approach to the case of relatively hyperbolic groups \cite{Osi10}.

\begin{thm}\label{sct}
Let $G$ be a group hyperbolic relative to a collection of subgroups
$\Hl $, $S$ a suitable subgroup of $G$, $W$ a finite subset of $G$. Then there exists an epimorphism $\eta \colon G\to \overline{G}$ such
that:
\begin{enumerate}
\item[(a)] The restriction
of $\eta $ to $\bigcup\limits_{\lambda \in \Lambda} H_\lambda $ is injective. Henceforth we think of $H_\lambda$ as a subgroup of $\overline{G}$.
\item[(b)] $\overline{G}$ is hyperbolic relative to $\Hl$.
\item[(c)] $\eta (W)\subseteq \eta (S)$.
\item[(d)] $\eta(S)$ is suitable in $\overline{G}$.
\item[(e)] If $G$ is torsion free, then so is $\overline{G}$.
\item[(f)] Every $H_\lambda$ is a proper subgroup of $\overline{G}$ and $\overline{G}$ is not virtually cyclic.
\item[(g)] $K(\overline{G})=\{1\}$.
\end{enumerate}
\end{thm}

\begin{proof}
Parts (a)--(e) were proved in \cite[Theorem 2.4]{Osi10}. Parts (f) and (g) follow from the existence of a suitable subgroup guaranteed by (d) (see Lemma \ref{AMO-suit}).
\end{proof}

The first two parts of the next theorem were proved in \cite{Osi07}. For details and relation to Thurston's theory of hyperbolic Dehn filling of $3$-manifolds we refer to \cite{Osi07}.

\begin{thm}\label{Df}
Let $G$ be a group hyperbolic relative to a collection of subgroups $\Hl $, $S$ a suitable subgroup of $G$. Then there exist finite subsets $\mathcal F_\lambda \subseteq H_\lambda\setminus \{ 1\}$ such that for every collection of subgroups $N_\lambda \lhd H_\lambda$ satisfying $N_\lambda \cap \mathcal F_\lambda =\emptyset $, $\lambda\in \Lambda$, the following conditions hold.
\begin{enumerate}
\item[(a)] For every $\lambda\in \Lambda $, we have $H_\lambda \cap N=N_\lambda$, where $N$ is the normal closure of $\bigcup_{\lambda\in \Lambda} N_\lambda$ in $G$. This is equivalent to the injectivity of the natural map $H_\lambda/N_\lambda \to G/N$. In what follows we think of $H_\lambda/N_\lambda $ as subgroups of $G/N$.
\item[(b)] The group $G/N$ is hyperbolic relative to the collection $\{H_\lambda/N_\lambda\}_{\lambda \in \Lambda}$.
\item[(c)] If $G$ and all quotient groups $H_\lambda/N_\lambda$ are torsion free, then so is $G/N$.

\item[(d)] The image of $S$ in $G/N$ is suitable.
\end{enumerate}
\end{thm}

\begin{proof}
As we already mentioned, the first two claims of the theorem were proved in \cite{Osi07}. To prove part (c) we note that $\Hl$ is hyperbolically embedded in $G$ in the terminology of \cite{DGO} by \cite[Proposition 4.28]{DGO}. Thus we can apply Theorem 7.19 from \cite{DGO} in our case. Let $\e$ denotes the natural homomorphism $G\to G/N$. Part (f) of Theorem 7.19 from \cite{DGO} states that every
element of $G/N$ acting elliptically (i.e., with bounded orbits) on $\Gamma(G/N, \e(X\cup \mathcal H))$ is an image of an element of $G$ acting elliptically on $\G$. Here by $\Gamma (H,Y)$ we denote the Cayley graph of a group $H$ with respect to a generating set $Y$. Thus if $\bar g\in G/N$ has finite order, there is a preimage $g\in G$ of $\bar g$ that acts elliptically on $\G$. However, according to \cite[Theorem 1.14]{Osi06} every such an element $g$ has finite order or is conjugate to an element of some $H_\lambda $. In the former case we get $g=1$ as $G$ is torsion free; hence $\bar g=1$. In the latter case we again obtain $\bar g=1$ as $H_\lambda/N_\lambda $ is torsion free.

The last claim of Theorem \ref{Df} can be derived from parts (a) and (b) as follows. By Lemmas \ref{AMO-suit} and \ref{KG}, there exists a loxodromic elements $g\in S$ such that $E_G(g)=\langle g\rangle $. By the definition of a suitable subgroup, there must be another loxodromic element $f\in S$ not commensurable with $g$. Since $f$ and $g$ are non-commensurable, we can apply Corollary \ref{Eg} twice and conclude that $G$ is hyperbolic relative to $\Hl\cup \{E_G(f), E_G(g)\}$. Increasing the finite subsets $\mathcal F_\lambda$ if necessary, we can assume that parts (a) and (b) of the theorem hold for this extended collection of peripheral subgroups and normal subgroups $N_\lambda \lhd H_\lambda $ and $\{ 1\} \lhd E_G(f)$, $\{ 1\} \lhd E_G(g)$. Thus $G/N$ is hyperbolic relative to
\begin{equation}\label{HlE12}
\{ H_\lambda/N_\lambda \}_{\lambda \in \Lambda }\cup \{ E,D\},
\end{equation}
where $E$, $D$ are the isomorphic images of $E_G(f)$ and $E_G(g)$.

Let $T$ be the image of $S$ in $G/N$. Note that $E\cap D$ is trivial by Lemma \ref{maln} applied to the peripheral subgroups $E$ and $D$ of $G/N$. As $T$ contains infinite virtually cyclic subgroups $E,D\le T$ with trivial intersection, $T$ cannot be virtually cyclic. Further if $K$ is a finite subgroup of $G/N$ normalized by $T$, then some finite index subgroup of $E$ centralizes $K$. Hence $K\le E$ by Lemma \ref{maln} applied to the peripheral subgroup $E$. Consequently, $K=\{1\}$ as $E$ is infinite cyclic.

By Lemma \ref{240}, $G/N$ is also hyperbolic relative to $\{ H_\lambda/N_\lambda \}_{\lambda \in \Lambda }$. Let $h\in E$ denote the image of $g$ in $G/N$. By Lemma \ref{maln} applied to the collection (\ref{HlE12}), the element $h$ is not conjugate to an element of any $H_\lambda/N_\lambda$. This means that $h$ is loxodromic with respect to $\{ H_\lambda/N_\lambda \}_{\lambda \in \Lambda }$. Thus $T$ contains a loxodromic (with respect to $\{ H_\lambda/N_\lambda \}_{\lambda \in \Lambda }$) element, is not virtually cyclic, and does not normalize any non-trivial finite subgroup of $G/N$. Hence $T$ is suitable in $G/N$ by Lemma \ref{AMO-suit}.
\end{proof}

\subsection{Proof of Theorem \ref{main3}}

We will derive Theorem \ref{main3} from Proposition \ref{Pij} and the following algebraic result in the same manner as we derived Theorem \ref{main1} from Corollaries \ref{bisr} and \ref{cor-com}. We say that two sequences  $\mathcal G=\{ (G_i,X_i)\}$, $\mathcal H=\{ (H_i,X_i)\}$ of groups $G_i$, $H_i$ and their generating sets $X_i$, $Y_i$ are \emph{asymptotically isomorphic} (written $\mathcal G\cong _{as} \mathcal H$) if for all but finitely many $i$ there exist isomorphisms $G_i\to H_i$ that take $X_i$ to $Y_i$. It is clear that the property of having infinitesimal spectral radius is preserved by asymptotic isomorphisms.

\begin{prop}\label{limrh-prop}
Let
\begin{equation}\label{colP}
\mathcal P=\{ (P_{ij},X_{ij})\mid (i,j)\in \mathbb N\times \mathbb N\}
\end{equation}
be a collection of groups $P_{ij}$ and their finite generating sets $X_{ij}$ such that:
\begin{enumerate}
\item[(P$_1$)] $|X_{ij}|=i$ for any $i,j\in \mathbb N$;
\item[(P$_2$)] for every $i\in \mathbb N$, $\lim_{j\to \infty}{\rm girth} (P_{ij},X_{ij})= \infty$;
\item[(P$_3$)] for every $i,j\in \mathbb N$, $P_{ij}$ has non non-cyclic free subgroups.
\end{enumerate}
Let also $H$ be a non-virtually cyclic  hyperbolic group and let $C$ be a countable group without non-cyclic free subgroups. Then there exists a quotient group $G$ of $H$, a family $\{Y_i\}_{i\in \mathbb N}$ of subsets $Y_i\subseteq G$ of cardinality $|Y_i|=i$,  and functions $u\colon G\times \mathbb N\to G$ and $j\colon G\times \mathbb N\to \mathbb N$ with the following properties.
\begin{enumerate}
\item[(a)] $C$ embeds in $G$.
\item[(b)] $G$ has no non-cyclic free subgroups.
\item[(c)] If $H$, $C$ and all $P_{ij}$ are torsion free (respectively, if $C$ and all $P_{ij}$ are torsion), then so is $G$.
\item[(d)] Given $g\in G\setminus \{ 1\}$, let
    $$
    T_{g,i}=\{ u(g,i)ygy^{-1} \mid y\in Y_i\}
    $$
    and let $H_{g,i}$ be the subgroup of $G$ generated by $T_{g,i}$. Then for every $g\in G\setminus\{ 1\}$, we have $\{ (H_{g,i}, T_{g,i})\}\cong_{as}\{ (P_{ij(g,i)}, X_{ij(g,i)})\}$.
\end{enumerate}
\end{prop}

To prove this proposition, we need a couple of lemmas.
The first one will be used to deal with the torsion free case.

\begin{lem} \label{GtoQ}
Let (\ref{colP}) be a collection of torsion free groups and generating sets satisfying (P$_1$) and (P$_2$). Let $G$ be a non-cyclic torsion free finitely generated group hyperbolic relative to a collection of proper subgroups $\Hl$. Suppose that for every $\lambda\in \Lambda$, $H_\lambda$ has no non-cyclic free subgroups. Then for any finite subset $A\subseteq G\setminus\{1\}$, any free subgroup $R\le G$, and any $n\in \mathbb N$, there exists an epimorphism $\e\colon G\to Q$ such that the following conditions hold.
\begin{enumerate}
\item[(Q$_1$)] The restriction of $\e$ to $H_\lambda$ is injective for all $\lambda \in \Lambda$. Henceforth we think of $H_\lambda$ as subgroups of $Q$.
\item[(Q$_2$)] There exist $k\in \mathbb N$, $x_1, \ldots, x_n\in Q$, and a set of elements $\{y_a, z_a\in Q\mid a\in A\}$,  such that for any $a\in A$, the set
    $$
    T_a=\{ y_ax_i\e(a)x_i^{-1}z_a \mid i=1, \ldots, n\}
    $$
    admits a bijection to $X_{nk}$ that extends to an isomorphism $\la T_a\ra \to P_{nk}$.
\item[(Q$_3$)] $Q$ is hyperbolic relative to $\Hl \cup \{ \la T_a\ra \}_{a\in A}$.
\item[(Q$_4$)] Either $R$ is cyclic or the restriction of $\e$ to $R$ is surjective.
\item[(Q$_5$)] $Q$ is torsion free.
\end{enumerate}
\end{lem}

\begin{proof}
Let
\begin{equation}\label{P2}
P=G\ast F(X)\ast F(Y)\ast F(Z),
\end{equation}
where $X=\{ \bar x_1, \ldots, \bar x_n\}$, $Y=\{ \bar y_a| a\in A\}$, $Z=\{ \bar z_a\mid a\in A\}$. For $a\in A$, let
$$
\overline T_a=\{ \bar y_a\bar x_ia\bar x_i^{-1}\bar z_a\mid i=1, \ldots, n\}
$$
and let $F$ be the subgroup of $P$ generated by $\overline T=\bigcup_{a\in A} \overline T_a$.

By Corollary \ref{mfs}, $F$ is free with basis $\overline T$ and $P$ is hyperbolic relative to $\{G, F\}$. Note that  $F=\ast_{a\in A} F_a$, where $F_a=\langle \overline T_a\rangle$. It follows immediately from the definition of relative hyperbolicity that $F$ is hyperbolic relative to $\{ F_a\}_{a\in A}$. Also $G$ is hyperbolic relative to $\Hl$ by our assumption. Applying Lemma \ref{240} several times, we obtain that $P$ is hyperbolic relative to the peripheral collection
\begin{equation}\label{HlFa}
\Hl \cup \{ F_a\}_{a\in A}.
\end{equation}

Note that if $R$ is not cyclic, then it is a suitable subgroup of $P$ (with respect to (\ref{HlFa})) by Corollary \ref{Ksuit}. Indeed it is clear that $R$ is not conjugate to any subgroup of $F_a$. Similarly it cannot be conjugate to a subgroup of any $H_\lambda$ since $H_\lambda$ does not contain non-cyclic free subgroups by our assumption. Similarly $G$ is a suitable subgroup of $P$ by Corollary \ref{Ksuit}. We let $S=R$ if $R$ is non-cyclic and $S=G$ otherwise. In both cases $S$ is suitable in $P$ with respect to (\ref{HlFa}).

Let $\mathcal F_a \subseteq F_a\setminus\{ 1\}$ and $\mathcal F_\lambda \subseteq H_\lambda\setminus\{ 1\}$  be the finite sets provided by Theorem \ref{Df} applied to the relatively hyperbolic group $P$ with peripheral collection (\ref{HlFa}) and the suitable subgroup $S$. Recall that $|\overline T_a|=n=|X_{nj}|$ for all $j\in \mathbb N$. By (P$_2$) there exists $k\in \mathbb N$ with the following property: for every $a\in A$, a bijection $\overline T_a\to X_{nk}$ extends to a homomorphism $F_a\to P_{nk}$ such that the kernel of this homomorphism, denoted $N_a$, does not intersect $\mathcal F_a$. Let $G_1$ be the quotient group of $P$ obtained as in Theorem \ref{Df} applied to the peripheral collection (\ref{HlFa}) and normal subgroups $\{1\} \lhd H_\lambda$ for $\lambda\in \Lambda$ and $N_a\lhd F_a$ for $a\in A$. Since the restriction of the natural homomorphism $P\to G_1$ to $H_\lambda$ is injective for every $\lambda\in \Lambda$, we keep the notation $H_\lambda$ for the image of $H_\lambda$ in $G_1$. Thus $G_1$ is hyperbolic relative to the collection
\begin{equation}\label{HlFaq}
\Hl \cup \{ F_a/N_a\cong P_{nk}\}_{a\in A},
\end{equation}
where the isomorphism $F_a/N_a\cong P_{nk}$ takes $\overline T_a$ to $X_{nk}$.

By part (d) of Theorem \ref{Df}, the image $S_1$ of $S$ in $G_1$ is a suitable subgroup of $G_1$ with respect to the peripheral collection (\ref{HlFaq}). Since $G$ is finitely generated, so are $P$ and $G_1$. Let $W$ be a finite generating set of $G_1$. Let $Q$ be  the quotient group of $G_1$ obtained by applying Theorem \ref{sct} to the relatively hyperbolic group $G_1$, finite subset $W$, and the suitable subgroup $S_1$.

Let $\e$ denote the composition of the natural embedding $G\hookrightarrow P$ and the natural homomorphisms $P\to G_1\to Q$. By part (c) of Theorem \ref{sct}, the restriction of $\e$ to $S$ is surjective. In particular, $\e $ is an epimorphism and we obtain (Q$_4$). We define $x_i=\e (\bar x_i)$, $y_a=\e (\bar y_a)$, $z_a=\e (\bar z_a)$. Properties (Q$_1$)--(Q$_3$) follow immediately from the construction of $G_1$ and parts (a), (b) of Theorem \ref{sct}. Finally note that $P$ is torsion free as so is $G$. Therefore $G_1$ is torsion free by part (c) of Theorem \ref{Df}. Consequently $Q$ is torsion free by part (e) of Theorem \ref{sct}.
\end{proof}

To deal with the torsion case, we need a slightly modified version.

\begin{lem} \label{GtoQ-tor}
Let (\ref{colP}) be a collection of torsion groups and generating sets satisfying (P$_1$) and (P$_2$). Let $G$ be a finitely generated group, which is hyperbolic relative to a collection of proper torsion subgroups $\Hl$. Assume that $G$ is not virtually cyclic and $K(G)=\{ 1\}$. Then for any finite subset $A\subseteq G\setminus\{1\}$, any element $t\in G$, and any $n\in \mathbb N$, there exists an epimorphism $\e\colon G\to Q$ satisfying (Q$_1$)--(Q$_3$) and the following condition.
\begin{enumerate}
\item[(Q$_4^\ast$)] $\e(t)$ has finite order.
\end{enumerate}
\end{lem}
\begin{proof}
We proceed in three steps. First we define $P$, $\overline T_a$, and $F_a$ as in the proof of Lemma \ref{GtoQ} and repeat the first step of that proof. As a result, we obtain a quotient group $G_1$ of $P$ hyperbolic relative to the collection
\begin{equation}\label{HlFa1}
\Hl \cup \{ F_a/N_a\cong P_{nk}\}_{a\in A},
\end{equation}
where the isomorphisms $F_a/N_a\cong P_{nk}$ are induced by bijections $\overline T_a\to X_{nk}$, as in Lemma \ref{GtoQ}.

By Corollary \ref{loxex}, $G$ contains an element $g$ which is loxodromic with respect to  $\Hl$.  Obviously every element of $G$ that is loxodromic with respect to $\Hl$ remains loxodromic in $P$ with respect to (\ref{HlFa}). Since $K(G)=\{1\}$, $G$ does not normalize any non-trivial finite subgroup of $P$. Hence $G$ is suitable in $P$ with respect to the collection (\ref{HlFa}) by Lemma \ref{AMO-suit}. By part (d) of Theorem \ref{Df} we can assume that the image of $G$ in $G_1$, which we denote by $H$, is suitable in $G_1$ with respect to the collection (\ref{HlFa1}).

The second step is also similar to the one from the proof of Lemma \ref{GtoQ}. Let $W$ be a finite generating set of $G_1$. We apply Theorem \ref{sct} to the relatively hyperbolic group $G_1$ with peripheral collection (\ref{HlFa1}), finite subset $W$, and suitable subgroup $S=H$. The resulting quotient of $G_1$ is also a quotient of $G$ by part (c) of Theorem \ref{sct}. We denote this quotient group by $G_2$. By parts (a) and (b) of Theorem \ref{sct}, $G_2$ is hyperbolic relative to the natural image of the collection (\ref{HlFa1}). Since the map $G_1\to G_2$ is injective on peripheral subgroups, by a slight abuse of notation we can think of (\ref{HlFa1}) as the collection of peripheral subgroups of $G_2$.

Let $s$ denote the image of $t$ in $G_2$. If $s$ already has finite order in $G_2$, we let $Q=G_2$.  If $s$ as infinite order, then $s$ is necessarily loxodromic as all subgroups $H_\lambda$ and $F_a/N_a\cong P_{nk}$ are torsion. By Corollary \ref{Eg}, $s$ is contained in a virtually cyclic subgroup $E_{G_2}(s)$ such that $G_2$ is hyperbolic relative to
\begin{equation}\label{HlFa2}
\Hl \cup \{ F_a/N_a\cong P_{nk}\}_{a\in A}\cup \{ E_{G_2}(s)\},
\end{equation}
Since $E_{G_2}(s)$ is virtually cyclic, $\la s\ra$ has finite index in  $E_{G_2}(s)$. Hence there exists $m\in \mathbb N$ such that $\la s^m\ra$ is normal in $E_{G_2}(s)$. Passing to a multiple of $m$ if necessary, we can ensure that $\la s^m\ra$ misses any fixed finite subset of $G_2$. Hence we can apply Theorem \ref{Df} to the group $G_2$ with peripheral collection (\ref{HlFa2}) and normal subgroups $\{ 1\}\lhd H_\lambda$ for $\lambda \in \Lambda$, $\{ 1\}\lhd F_a/N_a$ for $a\in A$, and $\la s^m\ra\lhd E_{G_2}(s)$. Let $Q=G_2/N$ be the resulting quotient group, where $N$ is the normal closure of $s^m$ in $G_2$.

Obviously the natural image of $t$ has finite order in $Q$. As in the proof of Lemma \ref{GtoQ}, deriving conditions  (Q$_1$)--(Q$_3$) from Theorems \ref{sct} and \ref{Df} is straightforward.
\end{proof}

\begin{proof}[Proof of Proposition \ref{limrh-prop}]
Without loss of generality we can assume that $C$ is non-cyclic. Recall that every countable torsion group embeds in a finitely generated torsion group \cite{Phi}. Similarly every countable torsion free group without non-cyclic free subgroups embeds in a finitely generated torsion free group without non-cyclic free subgroups; this is an immediate consequence of the main theorem of \cite{OO}. Thus we can also assume that $C$ is finitely generated.

Recall that every hyperbolic group $H$ contains a maximal finite normal subgroup $K(H)$ \cite{Ols93}. Since hyperbolicity is a quasi-isometric invariant, $H/K(H)$ is also hyperbolic and not virtually cyclic. Thus passing to $H/K(H)$ if necessary, we can assume that $K(H)=\{1\}$. In this case, $H\ast C$ is hyperbolic relative to $C$ and the subgroup $H$ is suitable in $H\ast C$ by Lemma \ref{AMO-suit}. Applying Theorem \ref{sct} to the relatively hyperbolic group $H\ast C$ with peripheral collection $\{ C\}$, a finite generating set $W$ of $C$, and the suitable subgroup $S=H$, we obtain a quotient group $G_0$ of $H$ such that $C$ embeds in $G_0$ as a proper subgroup, $G_0$ is hyperbolic relative to $C$, and $G_0$ is torsion free whenever $H$ and $C$ are.

The desired group $G$ will be the limit of a sequence of groups and epimorphisms
\begin{equation}\label{indlim}
H\to G_0\stackrel{\e_1}\longrightarrow G_1 \stackrel{\e_2}\longrightarrow G_2\stackrel{\e_3}\longrightarrow\ldots
\end{equation}
constructed by induction. In what follows we denote by $\delta_n\colon G_0\to G_{n}$ the composition $\e_{n}\circ\cdots \circ \e_1$, by $X_0$ a finite generating set of $G_0$, and by $X_{n}$ the natural image of $X_0$ in $G_{n}$.

We begin with the torsion free case. Thus we assume that $H$, $C$, and all groups in $\mathcal P$ are torsion free. Let $\{ F_1, F_2, \ldots \}$ be the set of all free subgroups of $G_0$ of rank $2$. Suppose that we have already constructed a quotient group $G_{n-1}$ for some $n>0$, which is torsion free and hyperbolic relative to a certain family of proper subgroups $\Hl$ such that no $H_\lambda$ has non-cyclic free subgroups (for $i=0$, we have $\Hl=\{C\}$; the peripheral collection will increase at every step). Let $G_{n}$ be the quotient group obtained by applying Lemma \ref{GtoQ} to the relatively hyperbolic group $G_{n-1}$ with peripheral collection $\Hl$, finite set
\begin{equation}\label{Adef}
A=\{ a\in G_{n-1} \mid 0<|a|_{X_{n-1}}\le n\} ,
\end{equation}
and the free subgroup $R=\delta_{n-1}(F_{n})$. Let $\e_{n}\colon G_{n-1}\to G_{n}$ denote the corresponding epimorphism. The peripheral collection of $G_n$ is given by (Q$_3$). Note that the new subgroup $\langle T_a\rangle$ in this peripheral collection is isomorphic to some $P_{ij}$ and hence does not contain non-cyclic free subgroups by (P$_3$). The other inductive assumptions for $G_{n}$ follow immediately from (Q$_1$)--(Q$_3$) and (Q$_5$).

Let $G$ be the limit of the sequence (\ref{indlim}). That is, $G=G_0/N$, where $N=\bigcup\limits_{i=1}^\infty {\rm Ker}\, \delta_i $. Let $X$ denote the image of $X_0$ in $G$. Obviously $G$ is generated by $X$. Properties (Q$_1$) and (Q$_3$) imply by induction that $C\cap {\rm Ker}\,\delta_n=\{1\}$ for every $n$. Hence $C\cap N=\{1\}$ and thus $C$ embeds in $G$. Furthermore, if $K$ is a non-cyclic free subgroup of $G$, then it contains a free subgroup $F\le K$ of rank $2$. Let $F_n$ be a preimage of $F$ in $G_0$. As the isomorphism $F_n\to F$ factors through $\delta_{n-1}(F_n)$, the latter subgroup is non-cyclic. Hence by (Q$_4$) applied at step $n$ we have $\delta_n(F_n)=G_n$ and hence $F=G$. However this contradicts the fact that $C$ embeds in $G$ (recall that we assume $C$ to be non-cyclic). Thus $G$ has no non-cyclic free subgroups. It is also clear that $G$ is torsion free as so are all groups $G_n$ by (Q$_5$). These arguments prove (a)--(c).

It remains to show that $G$ satisfies (d). At step $n$ of the inductive construction, we have a subset $\{ x_1, \ldots, x_n\} \subseteq G_n$ and elements $y_a,z_a\in G_n$ provided by (Q$_2$)  (here $a$ ranges in the subset $A\subseteq G_{n-1}$ given by (\ref{Adef})).

Fix any non-trivial element $g\in G$. Let $n\ge |g|_X$ be a natural number and let $a$ be a preimage of $g$ in $G_{n-1}$ of length $|a|_{X_{n-1}}=|g|_X\le n$. Then (Q$_2$) and (\ref{Adef}) guarantee that there exists a bijection $$
T_a=\{ y_ax_i\e_n(a)x_i^{-1}z_a \mid i=1, \ldots,n \}\to X_{nk(n)}
$$
for some $k(n)$ that extends to an isomorphism $\la T_a\ra \to P_{nk(n)}$. Since at $n$th step of the inductive construction $\la T_a\ra$ becomes a peripheral subgroup, it remains untouched by the subsequent factorizations according to (Q$_1$). Hence the natural map from $\la T_a\ra$ to $G$ is injective. Denote by $Y_n$ the image of $\{ x_1, \ldots, x_n\}$ in $G$. For $n\ge |g|_X$, we also let $u(g,n)$ be the image of $z_ay_a$ in $G$ and $j(g,n)=k(n)$; for $n<|g|_X$ we define $u(g,n)$ and $j(g,n)$ arbitrarily. Then the set
$$
T_{g,n}=\{ u(g,i)ygy^{-1} \mid y\in Y_n\}
$$
is the image of
$$
z_aT_az_a^{-1}=\{ z_ay_ax_i\e_n(a)x_i^{-1} \mid i=1, \ldots,n \}
$$
in $G$ and for every $n\ge |g|_X$ we have a sequence of natural isomorphisms
$$
\langle T_{g,n}\rangle \cong \la T_a\ra \cong P_{nj(g,n)}.
$$
Their composition gives an isomorphism $\langle T_{g,n}\rangle\cong P_{nj(g,n)}$ that sends $T_{g,n}$ to $X_{nj(g,n)}$. This finishes the proof in the torsion free case.

The proof for torsion groups is similar. We start with a collection $\mathcal P$ of torsion groups and enumerate all elements of $G_0=\{ 1=g_0, g_1, \ldots \}$. Suppose that $G_{n-1}$ is already constructed for some $n>0$ and is hyperbolic relative to a certain collection of torsion subgroups $\Hl$. Define $G_{n}$ to be the group obtained by applying Lemma \ref{GtoQ-tor} to the relatively hyperbolic group $G_{n-1}$ with peripheral collection $\Hl$, the element $t=\delta_{n-1}(g_n)$ (we let $\delta_0=id_{G_0}$), and the set $A$ defined by (\ref{Adef}). As above, verification of the inductive assumptions for $G_n$ is straightforward.

Again let $G$ be the limit of the sequence (\ref{indlim}). Claims (a) and (d) are proved exactly as in the torsion free case. Further let $g\in G$ and let $g_n$ be a preimage of $g$ in $G_0$. Then the image of $g_n$ in $G_n$ becomes of finite order by (Q$_4^\ast$). Hence $|g|<\infty$. This gives (c). Clearly there is no need to prove (b) in the torsion case.
\end{proof}

We are now ready to prove Theorem \ref{main3}. The argument used below is the same as we used in the proof of Theorem \ref{main1} in Section 3.

\begin{proof}[Proof of Theorem \ref{main3}]
Let $\mathcal P$ be a collection of torsion (respectively, torsion free) groups given by Proposition \ref{Pij}. Let $G$ be the torsion (respectively, torsion free) group constructed in Proposition \ref{limrh-prop} and let $H_{g,i}$, $T_{g,i}$ be given by part (d) of Proposition \ref{limrh-prop}. Fix any $g\in G\setminus\{ 1\}$. Since $\lambda_G(u(g,i))$ is unitary,  we have
\begin{equation}\label{G1}
\left\|\sum\limits_{y\in Y_i}\lambda_G(ygy^{-1})\right\| = \left\|\lambda_G(u(g,i)) \sum\limits_{y\in Y_i}\lambda_{G}(ygy^{-1})\right\|= \left\|\sum\limits_{y\in Y_i}\lambda_{G}(u(g,i)ygy^{-1})\right\|.
\end{equation}
Further by part (a) of Lemma \ref{norms}, we obtain
\begin{equation}\label{G2}
\left\|\sum\limits_{y\in Y_i}\lambda_{G}(u(g,i)ygy^{-1})\right\| = \left\|\sum\limits_{y\in Y_i}\lambda_{H_{g,i}}(u(g,i)ygy^{-1})\right\| = \left\|\sum\limits_{t\in T_{g,i}}\lambda_{H_{g,i}}(t)\right\| .
\end{equation}
By Proposition \ref{limrh-prop} (d) and Proposition \ref{Pij} (c), the sequence $\{ (H_{g,i}, T_{g,i})\}_{i\in \mathbb N}$ has infinitesimal spectral radius. Combining this with (\ref{G1}) and (\ref{G2}) yields
$$
\lim_{i\to \infty} \frac1{|Y_i|}\left\|\sum\limits_{y\in Y_i}\lambda_G(ygy^{-1})\right\|=0.
$$
Thus Lemma \ref{AL} applies and $C_{red}^\ast (G)$ is simple with unique trace.
\end{proof}

Now we derive Corollary \ref{main2}.

\begin{proof}[Proof of Corollary \ref{main2}]
Recall that there are $2^{\aleph_0}$ pairwise non-isomorphic finitely generated torsion groups. This fact immediately follows from the main result of \cite{Ols79} or from Grigorchuk's results about growth functions of torsion groups \cite{Gri}. It can also be derived from the embedding theorem proved by Phillips in \cite{Phi} or from results about so-called $\Pi$-graded groups obtained in the joint paper of the authors \cite{OO08} (the latter paper is probably the most elementary).

On the other hand, every countable group has only countably many finitely generated subgroups. Combining these facts with Theorem \ref{main3}, we obtain that every non-virtually cyclic  hyperbolic group has continuously many torsion quotients whose reduced $C^\ast$-algebra is simple and has unique trace. In particular, we obtain part (a) of Corollary \ref{main2}.

The same argument works in the torsion free case as there are continuously many finitely generated torsion free groups without free subgroups (one can take these groups to be solvable of derived length $3$, see \cite{Hall})
\end{proof}

Recall that the reduced $C^\ast$-algebra of a non-virtually cyclic hyperbolic group $H$ is simple and has unique trace if and only if $H$ contains no non-trivial finite normal subgroups; an analogous result also holds for relatively hyperbolic groups \cite{AM,Har88}. This obviously implies that the group $G$ constructed in the proof of Proposition \ref{limrh-prop} are limits of $C^\ast$-simple relatively hyperbolic groups. Similarly using results of \cite{book}, it is not hard to show that the free Burnside groups $B(m,n)$ of large odd exponent and rank $m\ge 2$ are limits of $C^\ast$-simple hyperbolic groups $G(i)$ (the statement about hyperbolicity of these groups is written down explicitly in \cite{Iva}). We note that these facts alone are not sufficient to derive $C^\ast$-simplicity and uniqueness of trace.

\begin{ex}\label{non-simple}
It was noted in \cite{Osi02} that there exists a sequence of hyperbolic groups and epimorphisms $H_1\to H_2\to \ldots$ that converges to the wreath product $G=\mathbb Z_2 {\, \rm wr\,} \mathbb Z$, where
$$
H_n=\left\langle t, a_{-n}, \ldots, a_n\; \; \left|  \; \;
\begin{array}{c}
t^{-1}a_kt=a_{k+1}, \;\; k=-n, \ldots, n-1 \\
a_i^2=1, \; [a_i,a_j]=1, \;\; i,j= -n, \ldots , n \\
\end{array}
\right.\right\ra
$$
It is easy to see that every $H_n$ is an HNN-extensions of a finite abelian group
$$
A_n= \la a_{-n}, \ldots, a_n\mid a_i^2=1, \; [a_i,a_j]=1, \; i,j= -n, \ldots , n \ra
$$
with stable letter $t$. Hence $H_n$ is virtually free and, in particular, hyperbolic. It is also easy to see that
\begin{equation}\label{Aint}
t^{-2n-1}A_nt^{2n+1}\cap A_n=\{ 1\}.
\end{equation}
Recall that every finite group acting on a tree without inversions must fix a vertex. If $N$ is a finite subgroup of $H_n$, then it fixes a vertex of the Bass-Serre tree $T$ associated to the HNN-extension structure of $H_n$. If $N$ is also normal, then it must fix all $T$ as the action of $H_n$ on vertices of $T$ is transitive. On the other hand, (\ref{Aint}) means that the pointwise $H_n$-stabilizer of a pair of vertices of $T$ is trivial.  Therefore $H_n$ has no nontrivial finite normal subgroups. Thus $C^\ast_{red}(H_n)$ is simple and has unique trace for every $n\in \mathbb N$. However $G$ is amenable and hence $C^\ast_{red}(G)$ is not simple and has many traces. Similar examples of sequences converging to lacunary hyperbolic amenable groups can be found in \cite{OOS}.
\end{ex}

\vspace{1cm}

\noindent \textbf{Alexander Olshanskii: } Department of Mathematics, Vanderbilt University, Nashville 37240, U.S.A.\\
E-mail: \emph{alexander.olshanskiy@vanderbilt.edu}\\

\smallskip

\noindent \textbf{Denis Osin: } Department of Mathematics, Vanderbilt University, Nashville 37240, U.S.A.\\
E-mail: \emph{denis.osin@gmail.com}

\end{document}